\documentclass[a4paper]{amsart}
\usepackage[utf8]{inputenc} 
\usepackage{amssymb}
\usepackage[all,dvips]{xy}
\usepackage{verbatim}
\usepackage{setspace}
\usepackage{longtable}
\usepackage{array}
\usepackage{hyperref}

\theoremstyle{plain}
\newtheorem{thm}{Theorem}[section]
\newtheorem{cor}[thm]{Corollary}
\newtheorem{lem}[thm]{Lemma}
\newtheorem{prop}[thm]{Proposition}

\theoremstyle{definition}
\newtheorem{defn}[thm]{Definition}
\newtheorem{remark}[thm]{Remark}

\newtheorem*{dank}{Acknowledgments}

\numberwithin{equation}{section}

\newcommand\calA{\mathcal{A}}            
            \newcommand\calO{\mathcal{O}}

\newcommand\calF{\mathcal{F}}            \newcommand\calS{\mathcal{S}}
\newcommand\calG{\mathcal{G}}            
            \newcommand\calU{\mathcal{U}}
            \newcommand\calV{\mathcal{V}}
            \newcommand\calW{\mathcal{W}}
\newcommand\calK{\mathcal{K}}            
\newcommand\calL{\mathcal{L}}

\newcommand\id{\mathrm{id}}
\newcommand{\GL}[1]{\mathrm{GL}_{#1}}
\newcommand{\PR}[1]{\mathbb{P}^{#1}}
\newcommand{\FF}[1]{\mathbb{F}_{#1}} \newcommand{\Flag}{\mathrm{Flag}}
\newcommand{\OFlag}{\mathrm{Flag}^\perp}
\newcommand{\AFlag}{\mathrm{Flag}^{\perp,F,V}}
\newcommand{\fk}{\mathbb{F}}
\newcommand{\Dieu}{\mathbb{D}}
\newcommand{\Dieui}{\mathbb{D}_\infty}
\newcommand{\too}[1]{\overset{#1}{\to}}

\newcommand{\Adm}{\mathrm{Adm}(\mu)}
\newcommand{\Admn}{\mathrm{Adm}(\mu)^{(0)}}
\newcommand{\KR}[1]{\calA_{I,#1}}
\newcommand{\ES}{\mathbf{ES}}
\newcommand{\Wfin}{W_{\mathrm{final}}}
\newcommand{\pipe}{\mathrel{|}}
\newcommand{\pair}[2]{\left< #1,#2 \right>}
\newcommand{\paird}{\left<\cdot,\cdot\right>}
\DeclareMathOperator{\Spec}{Spec}

\DeclareMathOperator{\im}{im}
\DeclareMathOperator{\rk}{rank}
\begin{document}
\author{Philipp Hartwig}
\address{Universität Duisburg-Essen\\
Institut für experimentelle Mathematik\\
Ellernstr. 29\\
45326 Essen\\
Germany}
\email{philipp.hartwig@stud.uni-due.de}
\title[moduli spaces of abelian varieties of dimension 3]{On the reduction of the Siegel moduli space of abelian varieties of dimension 3 
with iwahori level structure}
\begin{abstract}
	We study the moduli space of abelian threefolds with Iwahori 
	level structure in positive characteristic. We explicitly determine the 
	fibers of the canonical projection to the moduli space of principally polarized 
	abelian varieties and draw conclusions about the relationship between the 
	Ekedahl-Oort, the Kottwitz-Rapoport and the Newton stratification on these 
	spaces.
\end{abstract}
\maketitle

\section{Introduction}
Fix a prime $p$, an integer $g\geq 1$ and an algebraic closure $\fk$ of 
$\FF{p}$. Denote by $\calA_g$ the moduli space of principally polarized abelian 
varieties of dimension $g$ over $\fk$ and by $\calA_I$ the moduli space of 
abelian varieties of dimension $g$ over $\fk$ with Iwahori level structure (see 
Section \ref{secnotation} for details). 

In this paper we determine an explicit description of the fibers of the canonical 
projection $\pi:\calA_I\to\calA_g$ in the case $g=3$ and use this description to 
study the relationship between the natural stratifications on $\calA_I$ and 
$\calA_g$.

On $\calA_g$ we have the
$p$-rank stratification which has the property that two abelian varieties lie in 
the same stratum if and only if their $p$-ranks coincide.  We have the 
Ekedahl-Oort stratification, originally defined in \cite{Oort}, which is given 
by the isomorphism type of the kernel of multiplication by $p$ on the abelian 
variety. There is an explicit bijection from the set of EO strata to the set of 
final sequences of length $g$, that is, to the set of maps 
$\psi:\{1.\ldots,2g\}\to \mathbb{N}$ with $\psi(0)=0,\ \psi(2g)=g$, such that
\begin{equation*}
	\psi(i)\leq\psi(i+1)\leq\psi(i)+1
\end{equation*}
and
\begin{equation*}
	\psi(i)<\psi(i+1)\Leftrightarrow \psi(2g-i)=\psi(2g-i-1)
\end{equation*}
for $0\leq i<2g$. If $\psi$ is a final sequence, we denote the corresponding EO 
stratum by $EO_\psi$. 

Trivially the EO stratification is a refinement of the 
$p$-rank stratification.

Furthermore we have the Newton stratification, given by the isogeny type of the 
Barsotti-Tate group of the abelian variety. We are primarily concerned with one 
special Newton stratum, namely the supersingular locus $\calS_g$. In general 
neither of the Newton or the EO stratification is a refinement of the other. In 
fact the supersingular locus is not a union of EO strata for $g\geq 3$. In 
\cite{Harashita} Harashita determines those EO strata that are entirely 
contained in the supersingular locus. 

For $g=3$ there are four EO strata of $p$-rank 0, totally ordered by their 
dimensions. By Harashita's result, the 0- and the 1-dimensional stratum are 
contained in the supersingular locus. Using a normal form for the Dieudonn\'e 
module of the Barsotti-Tate group of a supersingular abelian variety, due to 
Harashita, we prove that the 2-dimensional EO stratum is contained in the 
complement of $\calS_3$.  It then follows for dimension reasons that the 
3-dimensional EO stratum intersects the supersingular locus in an open dense 
subset of $\calS_3$.

On $\calA_I$ we have the Kottwitz-Rapoport stratification, given by the 
relative position of the chain of de Rham cohomology groups and the chain of 
Hodge filtrations associated with an element of $\calA_I$. There is an explicit 
bijection from the set of KR strata to the set of admissible elements $\Adm$, 
where the latter is a subset of the extended affine Weyl group of the group of symplectic 
similitudes $\mathrm{GSp}_{2g}$. There is a unique element $\tau$ of length 0 
such that $\Adm\subset W_a\tau$, where $W_a$ is the affine Weyl group of $G$, a 
Coxeter group generated by simple reflection $s_0,\ldots,s_g$ (described 
explicitly in Section \ref{GroThe}). 

We denote by $\calS_I$ the preimage of $\calS_g$ under $\pi$, which we also call 
the supersingular locus.
In \cite{KRDL} and \cite{SSLIW} Görtz and Yu study the dimension of $\calS_I$ 
and they determine those KR strata that are entirely contained in the 
supersingular locus. But again it is not true in general that $\calS_I$ is a 
union of KR strata and it is natural to ask which other KR strata have a 
nonempty intersection with the supersingular locus and what the dimension of 
this intersection is. Another question concerning the KR stratification deals 
with its relationship to the EO stratification. It is known that the image of a 
KR stratum under $\pi$ is always a union of EO strata but it is not known which 
EO strata occur in the image of a given arbitrary KR stratum. This question has 
been studied by Ekedahl and van der Geer in \cite{Geer} (cf. \cite[Section 
9]{SSLIW}) and also by Görtz and Hoeve in \cite{Hoeve} . 

To answer these questions for $g=3$ we need to investigate the fibers of $\pi$.  
Classical Dieudonn\'e theory provides us with an injective map from a fiber of 
$\pi$ into a suitable flag variety over $\fk$ and it can be shown that this map 
is actually a universally injective, finite morphism of algebraic varieties over 
$\fk$. In particular it induces a universal homeomorphism onto its image and 
in order to study topological properties of the fibers it is therefore 
sufficient to study their images under these respective maps. Up to isomorphism 
these images only depend on the EO stratum of the basepoint and hence there are 
only finitely many cases that have to be considered. If $\psi$ is a final 
sequence we denote this image of the fiber over a point of $EO_\psi(\fk)$ by 
$\AFlag_{\psi}=\AFlag_{\psi,2g}$. While the conditions that determine 
$\AFlag_\psi$ as a closed subvariety of a full flag variety over $\fk$ are easy 
to describe, the geometry of the resulting variety is rather complicated.

To give the reader an impression of what has to be expected let us sketch the 
geometry of the variety $\AFlag_{\psi_0}$, where $EO_{\psi_0}$ is the 
0-dimensional EO stratum, see Section \ref{FibI}. 

\begin{thm}
	Let $g=3$ and let $A\in\calA_3(\fk)$ be a \emph{superspecial} abelian 
	variety. Then there is a universal homeomorphism from the fiber 
	$\pi^{-1}(A)$ onto $\AFlag_{\psi_0}$. The variety $\AFlag_{\psi_0}$ is 
	decomposed into irreducible components
	\begin{equation*}
		Y\cup Z\cup \coprod_{\zeta\in \widetilde {I}} T_\zeta,
	\end{equation*}
	where $\widetilde{I}$ can be chosen as $\{(x:y:z)\in \PR{2}(\FF{p^2})\pipe 
	x^pz+y^{p+1}+xz^p=0\}$ and such that
	\begin{itemize}
		\item $Y$ is isomorphic to the variety of full flags in 
			$\fk^3$,
		\item $Z$ can be considered as a $\PR{1}_\fk$-bundle over a variety $Z_0$, 
			where $Z_0$ is itself a $\PR{1}_\fk$-bundle over the irreducible curve 
			$V_+(X_1^pX_3+X_2^{p+1}+X_1X_3^p)\subset \PR{2}_{\fk}$ (for homogenous 
			coordinates $X_1$, $X_2$ and $X_3$ on $\PR{2}_{\fk}$),
		\item  each $T_\zeta$ is 
			isomorphic to the blowing-up of $\PR{2}_\fk$ in a closed point.
	\end{itemize}
\end{thm}

Sticking to the notation of the Theorem, we see that $\dim Y=\dim Z=3$ and $\dim 
T_\zeta=2$. Furthermore the $T_\zeta$ are pairwise disjoint. The 
intersection $Y\cap Z$ is isomorphic to the variety $Z_0$. $Z$ intersects each 
$T_\zeta$ in its exceptional curve, while $Y\cap T_\zeta$ is a different 
subvariety of $T_\zeta$ isomorphic to $\PR{1}_\fk$. Finally the triple 
intersection $Y\cap Z\cap T_\zeta$ only consists of one point.

Concerning the fiber over abelian varieties of positive $p$-rank we prove the 
following general result (modelled on the ``shuffle construction'' explained in
\cite[5.2]{Yu}) that provides a method for reducing the case of 
positive $p$-rank to the case of $p$-rank 0 in lower dimensions, see Section 
\ref{secposprank}. 

\begin{prop}
	Let $g\geq 1$, $k\geq 0$ and let $A\in\calA_g(\fk)$ be of $p$-rank $k$. Let $\psi$ be the final 
	sequence with $A\in EO_\psi$.
	\begin{enumerate}
		\item Let $A$ be ordinary. Then the fiber over $A$ is 
			discrete and 
			\begin{equation*}
				\# \left(\pi^{-1}(A)\right)=ON_g:=
				2^g\#\Flag_g(\FF{p})=2^g\frac{\prod_{l=1}^g (p^l-1)}{(p-1)^g}.
			\end{equation*}
			Here we denote by $\Flag_{g}(\FF{p})$ the set of flags 
			$(\calF_j)_{j=0}^g$ in $(\FF{p})^g$ with $\dim \calF_j=j$ for all 
			$0\leq j\leq g$.
		\item Let $1\leq k\leq g-1$. Then $\AFlag_\psi$ is isomorphic to 
			$\begin{pmatrix}g\\k\end{pmatrix}ON_{k}$ disjoint copies of 
			$\AFlag_{\widetilde{\psi},2(g-k)}$, where $\widetilde{\psi}$ is the 
			final sequence of length $g-k$ determined by 
			$\widetilde{\psi}(i)=\psi(k+i)-k$ for $0\leq i\leq g-k$.
	\end{enumerate}
\end{prop}

This result will allow us determine the number of connected components of the 
fibers of $\pi$:

\begin{prop}
	Let $g\geq 1$ and $k\geq 0$. If $A\in\calA_g(\fk)$ is of $p$-rank $k$,
	the fiber $\pi^{-1}(A)$ consists of $\begin{pmatrix}g\\k\end{pmatrix}ON_{k}$ 
	connected components. In particular it is connected if and only if $k=0$.
\end{prop}

From the calculations of the varieties $\AFlag_\psi$ for $g=3$ it is rather easy 
to determine which KR strata intersect the fiber of $\pi$ over a given element of 
$\calA_3(\fk)$ and what the dimension of this intersection is. From this we can 
determine the EO strata which occur in the image of a given KR stratum: 

\begin{thm}[Section \ref{ESset}]
	For an element $x\in\Adm$ of $p$-rank 0 denote by $\ES(x)$ the 
	set of final sequences such that 
	$\pi(\KR{x})=\coprod_{\psi\in\ES(x)}EO_\psi$. Then the following table 
	contains a complete list of the sets $\ES(x)$ in the case $g=3$.
	\begin{center}
		\begin{tabular}{|l|cccc|}
			\hline
			$x$&\multicolumn{4}{c|}{$\ES(x)$}\\
			\hline
			\hline
			$\tau, s_1\tau, s_2\tau, s_{21}\tau, s_{12}\tau, s_{121}\tau$& 
			$\psi_0$&&&\\
			\hline
			$s_3\tau, s_0\tau$& &$\psi_1$&&\\
			\hline
			$s_{30}\tau$&$\psi_0$&$\psi_1$&&\\
			\hline
			\hline
			$s_{10}\tau, s_{23}\tau, s_{20}\tau, s_{31}\tau,s_{01}\tau, 
			s_{32}\tau$&&&$\psi_2$&\\
			\hline
			$s_{310}\tau, s_{320}\tau$&$\psi_0$&&$\psi_2$&\\
			\hline
			$s_{3120}\tau$&$\psi_0$&$\psi_1$&$\psi_2$&\\
			\hline
			$s_{120}\tau, s_{312}\tau,s_{201}\tau, s_{231}\tau$&&$\psi_1$&$\psi_2$&\\
			\hline
			$s_{010}\tau, s_{323}\tau, s_{301}\tau, s_{230}\tau$&&&&$\psi_3$\\
			\hline
			$s_{2301}\tau$&$\psi_0$&$\psi_1$&&$\psi_3$\\
			\hline
			$s_{3010}\tau, s_{3230}\tau$&&&$\psi_2$&$\psi_3$\\
			\hline
		\end{tabular}
	\end{center}
	Here $\psi_i$ denotes the final sequence corresponding to the $i$-dimensional EO 
	stratum of $p$-rank 0, $0\leq i\leq 3$.  
\end{thm}

The upper block of the table contains the supersingular 
elements.

With this table we can show that the inclusion
\begin{equation}
	\coprod_{\substack{x\in\Admn\\\KR{x}\subset\calS_I}}\KR{x}\subseteq\pi^{-1}\left(\coprod_{\substack{w\in\Wfin\\EO_w\subset 
	\calS_g}}EO_w\right),
\end{equation}
which is valid for every $g\geq 1$, is a proper inclusion for $g=3$, negatively 
answering a question posed in a preliminary version of \cite{KRDL}.

Finally we show that for $g=3$ we have $\dim (\KR{x}\cap \calS_I)= \dim\KR{x}-1$
for every KR stratum $\KR{x}$ with $\emptyset\subsetneq \KR{x}\cap\calS_I\subsetneq 
\KR{x}$.

\begin{dank}
	I sincerely want to thank Ulrich G\"ortz for introducing me to this subject and for 
	his permanent interest in my work. The help he provided during the writing 
	of this paper has been invaluable. I also want to thank Michael Rapoport and 
	Chia-Fu Yu for helpful comments.
\end{dank}
\section{Notation}\label{secnotation}

\subsection{Basic notation and moduli spaces}

We fix a prime $p$, an integer $g\geq 1$, an integer $N\geq 3$ coprime to $p$, 
an algebraic closure $\fk$ of $\FF{p}$ and a primitive $N$-th root of unity 
$\zeta_N$ in $\fk$. Let $\sigma:\fk\to\fk$ denote the 
Frobenius morphism. We consider the moduli space $\calA_g=\calA_{g,N}$ of 
principally polarized abelian varieties of dimension $g$ over $\fk$ with a 
symplectic level-$N$-structure with respect to $\zeta_N$. It is a 
quasi-projective scheme over $\fk$, irreducible of dimension $g(g+1)/2$. We will 
usually omit the principal polarization and the level structure from our 
notation. We denote by $\calS_g$ the supersingular locus inside $\calA_g$. It is 
a closed subset, equi-dimensional of dimension $\left[\frac{g}{4}\right]$ by 
\cite{LiOort}. 

On the other hand, we consider the moduli space $\calA_I$ of tuples
\begin{equation*}
	(A_0\too{\alpha} A_1\too{\alpha}\cdots\too{\alpha} 
	A_g,\lambda_0,\lambda_g,\eta),
\end{equation*}
where
\begin{itemize}
	\item each $A_i$ is a $g$-dimensional abelian variety over $\fk$,
	\item $\alpha$ is an isogeny of degree $p$,
	\item $\lambda_0$ and $\lambda_g$ are principal polarizations on $A_0$ and 
		$A_g$, respectively, such that $(\alpha^g)^*\lambda_g=p\lambda_0$,
	\item $\eta$ is a symplectic level-$N$-structure on $A_0$ with respect to 
		$\zeta_N$.
\end{itemize}
$\calA_I$ has pure dimension $g(g+1)/2$. We will often omit $\eta$ and even 
$\lambda_0,\lambda_g$ from the notation. 

We denote by $\pi:\calA_I\to\calA_g$ the morphism sending a point 
\begin{equation*}
	(A_0\too{\alpha} A_1\too{\alpha}\cdots\too{\alpha} 
	A_g,\lambda_0,\lambda_g,\eta)
\end{equation*}
to the point $(A_0,\lambda_0,\eta)$. It is proper and surjective.

Inside $\calA_I$ we have the supersingular locus $\calS_I$, given by 
$\pi^{-1}(\calS_g)$ as a closed subset. It is shown in \cite{SSLIW} that for $g$ 
even we have $\dim \calS_I=g^2/2$ and that $(g^2-g)/2\leq \dim\calS_I\leq 
(g^2-1)/2$ if $g$ is odd. However the supersingular locus $\calS_I$ is not 
equi-dimensional as soon as $g\geq 2$.

\subsection{\texorpdfstring{The $p$-rank stratification}{The p-rank stratification}}\label{prank}

Let $X$ be a topological space. We call a set-theoretical decomposition 
$X=\coprod_{i\in I}X_i$ of $X$ a \emph{stratification on $X$} if for all $i\in 
I$ the set $X_i$ is nonempty, locally closed and satisfies
$\overline{X_i}=\cup_{j\in J_i}X_j$ for some subset $J_i\subset I$.

Let $A$ be an abelian variety of dimension $g$ over $\fk$. For $n\in\mathbb{N}$ we 
denote by $A[n]$ the kernel of multiplication by $n$ on $A$. It is a finite 
group scheme of rank $n^{2g}$ over $\fk$. There is an integer $0\leq i\leq g$ 
with $A[p](\fk)\simeq (\mathbb{Z}/p\mathbb{Z})^i$, called the \emph{$p$-rank} of 
$A$. We denote by $\calA_g^{(i)}$ the subset of $\calA_g$ where the $p$-rank of 
the underlying abelian variety is $i$. Then 
$\calA_g=\cup_{i\in\mathbb{N}}\calA_g^{(i)}$ is a stratification on $\calA_g$ 
with $\overline{\calA_g^{(i)}}=\cup_{j\leq i}\calA_g^{(j)}$. Similarly we write 
$\calA_{I}^{(i)}=\pi^{-1}(\calA_g^{(i)})$, but these sets do not give rise to a 
stratification on $\calA_I$.

\subsection{\texorpdfstring{The $a$-number}{The a-number}}

Let $\alpha_p$ be the $\fk$-group scheme representing the functor $S\mapsto 
\{s\in\calO_S(S)\pipe s^p=0\}$ on the category of $\fk$-schemes. For an abelian 
variety $A$ over $\fk$ we write $a(A)=\dim_{\fk}\mathrm{Hom}(\alpha_p,A)$. This 
integer is called the \emph{$a$-number} of $A$.

\subsection{Group theoretic notation}\label{GroThe}

We denote by $\mathrm{G}=\mathrm{GSp}_{2g}$ the group of symplectic similitudes.
We consider it as a subgroup of $\GL{2g}$ with respect to the embedding induced 
by the alternating form given on the standard basis vectors 
$e_1\ldots,e_{2g}$ by $(e_i,e_j)\mapsto 0,\ 
(e_{2g+1-i},e_{2g+1-j})\mapsto 0$ and $(e_i,e_{2g+1-j})\mapsto \delta_{ij}$ for 
$1\leq i,j\leq g$. We use the Borel subgroup of upper triangular matrices and 
the maximal torus $T$ of diagonal matrices. We denote by $W$ the finite Weyl 
group of $\mathrm{G}$ which we consider as a subgroup of the finite Weyl 
group of $\GL{2g}$. If we identify the latter with $S_{2g}$ in the usual way, an 
element $w$ of $S_{2g}$ lies in $W$ if and only if $w(i)+w(2g+1-i)=2g+1$ for all
$1\leq i\leq 2g$. Similarly we identify $X_*(T)$ with the group 
$\{(a_1,\ldots,a_{2g})\in\mathbb{Z}^{2g}\pipe 
a_1+a_{2g}=a_2+a_{2g-1}=\cdots=a_g+a_{g+1}\}$. For an element 
$x=(x_1,\ldots,x_{2g})$ of $X_*(T)$ we also write $x(i)$ instead of $x_i$. $W$ 
is generated by the elements $s_1,\ldots,s_g$ given by $s_g=(g,g+1)$ and 
$s_i=(i,i+1)(2g+1-i,2g-i)$ for $1\leq i\leq g-1$. Inside $W$ we have the subset 
$W_{\text{final},g}$ of elements $w$ with $w(1)<w(2)<\cdots<w(g)$. 

We denote by $\widetilde{W}=W\ltimes X_*(T)$ the extended affine Weyl group of 
$\mathrm{G}$. For an element $\lambda\in X_*(T)$ we denote by $t^\lambda$ the 
corresponding element of $\widetilde{W}$. We denote by $s_0$ and $\tau$ the 
elements of $\widetilde{W}$ given by
$s_0=(1,2g) t^{(1,0,\ldots,0,-1)}$ and 
\begin{equation*}
	\tau=\begin{pmatrix}1&\cdots&g&g+1&\cdots&2g\\g+1&\cdots&2g&1&\cdots&g\end{pmatrix}t^{(1,\ldots,1,0,\ldots,0)}.
\end{equation*}
The affine Weyl group $W_a$ of $\mathrm{G}$ is the subgroup of $\widetilde{W}$ 
generated by $s_0,\ldots,s_g$. It is an infinite Coxeter group. Our choice of 
generators $s_0,\ldots,s_g$ gives rise to a length function $\ell$ and the 
Bruhat order $\leq$ on $W_a$. We write $s_{i_1\ldots i_n}$ instead of 
$s_{i_1}\cdots s_{i_n}$. 

\subsection{Convention}

Let $K$ be an algebraically closed field. A \emph{variety (over $K$)} is a reduced 
scheme of finite type over $\Spec K$. A \emph{subvariety} of a variety is a reduced 
subscheme. If we identify a variety $X$ with its set $X(K)$ of $K$-valued points we 
refer to the latter object as a \emph{classical} variety.

\section{Dieudonn\'e modules}

This section introduces our notation for the Dieudonn\'e modules associated with 
the $p$-torsion of a principally polarized abelian variety. The principal 
polarization induces an isomorphism from the Dieudonn\'e module onto its dual 
and hence an isomorphism between co- and contravariant Dieudonn\'e theory. For 
most of our statements it will therefore not matter which theory we use. For the 
few statements where it is of importance, we will use the \emph{contravariant} 
theory. We refer to \cite{Dem} and \cite{Oda} for proofs of the statements 
below.

Given a ring $R$, an endomorphism $\alpha:R\to R$ and an $R$-module $M$, an 
additive map $\phi:M\to M$ is called \emph{$\alpha$-linear} if $\phi(r\cdot 
m)=\alpha(r)\cdot \phi(m)$ for all $r\in R,\ m\in M$.

Let $g\geq 1$.

\subsection{\texorpdfstring{The Dieudonn\'e module of $A[p]$}{The Dieudonn\'e module of A[p]}}\label{DieModAp}

Let $A\in \calA_g(\fk)$ and denote by $\Dieu=\Dieu(A[p])$ the Dieudonn\'e 
module of $A[p]$. It is a $2g$-dimensional vector space over $\fk$, equipped 
with linear maps $F:\Dieu^{(p)}\to \Dieu$ and $V:\Dieu\to \Dieu^{(p)}$, called 
Frobenius and Verschiebung respectively, where $\Dieu^{(p)}$ denotes the base 
change $\Dieu\mathrel{\otimes_{\fk,\sigma}}\fk$. As $\sigma$ is an isomorphism 
we can identify $\Dieu^{(p)}$ with $\Dieu$ and we will henceforth consider $F$ 
as a $\sigma$-linear and $V$ as a $\sigma^{-1}$-linear map $\Dieu\to\Dieu$. The 
principal polarization $A\to A^\vee$ induces a non-degenerate, alternating 
pairing $\paird=\paird_A$ on $\Dieu$. $F$, $V$ and this pairing have the 
following properties:

\begin{prop}\label{FVandPair}
\begin{enumerate}
	\item\label{imker} $\im V=\ker F$ and $\im F=\ker V$.
	\item\label{pair} $\left<Fx,y\right>=\left<x,Vy\right>^p$ for 
		all $x,y\in\Dieu$.
\end{enumerate}
\end{prop}

For future reference we include the following

\begin{cor}\label{FandV}
	\begin{enumerate}
		\item\label{perp} For any subspace $W\subset\Dieu$ we have 
			$V(W^\perp)=F^{-1}(W)^\perp$.
		\item\label{imVsymp} $(\im V)^\perp = \im V$.
	\end{enumerate}
\end{cor}

\begin{proof}
	\begin{enumerate}
		\item Using Proposition \ref{FVandPair}(\ref{pair}) we have
			\begin{equation*}
				x\in F^{-1}(W)\Leftrightarrow \forall y\in 
				W^\perp\ \pair{Fx}{y}=0\Leftrightarrow \forall y\in 
				W^\perp\ \pair{x}{Vy}=0\Leftrightarrow x\in V(W^\perp)^\perp
			\end{equation*}	
			and the statement follows.
		\item By (\ref{perp}) and Proposition \ref{FVandPair}(\ref{imker}) we have
			\begin{equation*}
				(\im V)^\perp=V(\Dieu)^\perp=V(0^\perp)^\perp=F^{-1}(0)=\ker F=\im V.
			\end{equation*}
	\end{enumerate}
\end{proof}

\subsection{\texorpdfstring{The Dieudonn\'e module of $A[p^\infty]$}{The 
Dieudonn\'e module of A[p\^{}oo]}}

Let $A\in \calA_g(\fk)$. We denote by $A[p^\infty]=\cup_{n}A[p^n]$ the Barsotti-Tate group of 
$A$. It has height $2g$. Associated to $A[p^\infty]$ is the Dieudonn\'e module 
$\Dieui=\Dieu(A[p^\infty])$. It is a free module of rank $2g$ over the Witt ring 
$W(\fk)$ of $\fk$, equipped with linear maps $F_\infty:\Dieui^{(p)}\to \Dieui$ 
and $V_\infty:\Dieui\to \Dieui^{(p)}$, called Frobenius and Verschiebung 
respectively, where $\Dieui^{(p)}$ denotes the base change 
$\Dieui\mathrel{\otimes_{W(\fk),\sigma_W}}W(\fk)$. Here we denote by $\sigma_W$ 
the Frobenius morphism on $W(\fk)$. As $\sigma_W$ is an isomorphism we can 
identify $\Dieui^{(p)}$ with $\Dieui$ and we will henceforth consider $F_\infty$ 
as a $\sigma_W$-linear and $V_\infty$ as a $\sigma_W^{-1}$-linear map 
$\Dieui\to\Dieui$. The principal polarization $A\to A^\vee$ induces a perfect, 
alternating pairing $\paird_\infty=\paird_{\infty,A}$ on $\Dieui$. $F_\infty$, 
$V_\infty$ and this pairing have the following properties:

\begin{prop}\label{FWVWandPair}
\begin{itemize}
	\item $F_\infty V_\infty = V_\infty F_\infty =p\cdot\id$.
	\item $\left<F_\infty x,y\right>=\left<x,V_\infty y\right>^{\sigma_W}$ for all $x,y\in\Dieu$.
\end{itemize}
\end{prop}
The reduction of $(\Dieui,F_\infty,V_\infty,\paird_\infty)$ modulo $p$ is isomorphic to 
$(\Dieu,F,V,\paird)$.

\subsection{Supersingular Dieudonn\'e modules}\label{SSDieMod}

We recall a result by Harashita, see \cite[Section 3]{Harashita}. Let 
$A\in\calA_g(\fk)$ be supersingular. Then there exists a basis 
$(X_1,\ldots,X_g,\allowbreak Y_1,\ldots,Y_g)$ of $\Dieui$ over $W(\fk)$ such 
that
\begin{itemize}
	\item $\pair{X_i}{Y_j}_\infty=\delta_{ij},\quad \pair{X_i}{X_j}_\infty=0,\quad 
		\pair{Y_i}{Y_j}_\infty=0\quad\text{for }1\leq i,j\leq g$.
	\item Let $w=(\delta_{i,g+1-j})_{i,j}\in M^{g\times g}(W(\fk))$. There is an $\varepsilon\in 
		W(\FF{p^2})^\times$ with $\varepsilon=-\varepsilon^{\sigma_W}$ and a 
		strictly lower triangular matrix $T\in M^{g\times g}(W(\fk))$ satisfying 
		$Tw=(Tw)^t$, such that $F_\infty$ and $V_\infty$ admit the following descriptions with 
		respect to this basis:
		\begin{equation*}
			F_\infty=\begin{pmatrix}
				T&-p\varepsilon^{-1}w\\
				\varepsilon w&0
			\end{pmatrix},\quad
			V_\infty=\begin{pmatrix}
				0&-p\varepsilon^{-1}w\\
				\varepsilon w&wT^{\sigma_W^{-1}}w
			\end{pmatrix}.
		\end{equation*}
\end{itemize}
Reducing modulo $p$, we get a basis 
$(\overline{X}_1,\ldots,\overline{X}_g,\overline{Y}_1,\ldots,\overline{Y}_g)$ of 
$\Dieu(A[p])$ over $\fk$ such that
\begin{itemize}
	\item $\left<\overline{X}_i,\overline{Y}_j\right>=\delta_{ij},\quad 
		\left<\overline{X}_i,\overline{X}_j\right>=0,\quad 
		\left<\overline{Y}_i,\overline{Y}_j\right>=0\quad\text{for }1\leq 
		i,j\leq g$.
	\item Let $\overline{w}=(\delta_{i,g+1-j})_{i,j}\in M^{g\times g}(\fk)$. There is an $\overline{\varepsilon}\in 
		\FF{p^2}^\times$ with 
		$\overline{\varepsilon}=-\overline{\varepsilon}^{\sigma}$ and a strictly 
		lower triangular matrix $\overline{T}\in M^{g\times g}(\fk)$ satisfying 
		$\overline{T}w=(\overline{T}w)^t$, such that $F$ and $V$ admit the 
		following descriptions with respect to this basis:
		\begin{equation*}
			F=\begin{pmatrix}
				\overline{T}&0\\
				\overline{\varepsilon} \overline{w}&0
			\end{pmatrix},\quad
			V=\begin{pmatrix}
				0&0\\
				\overline{\varepsilon} 
				\overline{w}&\overline{w}\overline{T}^{\sigma^{-1}}\overline{w}
			\end{pmatrix}.
		\end{equation*}
\end{itemize}
We have $a(A)=g-\rk(\overline{T})$.

\section{The EO stratification}

This section contains the results about the EO stratification on $\calA_g$ that 
we are going to use.

\subsection{Final sequences}\label{FinSeq}

We recall a notion defined in \cite{Oort}. Let $g\in\mathbb{N}$. A final 
sequence (of length $g$) is a map $\psi:\{1.\ldots,2g\}\to \mathbb{N}$ with 
$\psi(0)=0,\ \psi(2g)=g$, such that
\begin{equation*}
	\psi(i)\leq\psi(i+1)\leq\psi(i)+1
\end{equation*}
and
\begin{equation*}
	\psi(i)<\psi(i+1)\Leftrightarrow \psi(2g-i)=\psi(2g-i-1)
\end{equation*}
for $0\leq i<2g$.
Let $\ES=\ES_g$ be the set of final sequences of length $g$. We will identify $\ES$ with 
$\Wfin=W_{\text{final},g}$ via the bijection $\Wfin\to\ES$ given by $w\mapsto 
\psi_w$ with \begin{equation*}
	\psi_w(i)=i-\#\{a\in\{1,\ldots,g\}\pipe w(a)\leq i\}
\end{equation*}
and $\psi_w(2g-i)=\psi_w(i)+g-i$ for $0\leq i\leq g$. Instead of $\psi$ we will 
also write $(\psi(1),\psi(2),\ldots,\psi(g))$ to denote a final sequence.

\subsection{The canonical filtration}

Let $A\in\calA_g(\fk)$. Consider the set $e$ of all finite words in the symbols 
$F$ and $\perp$. In \cite[Section 5]{Oort}, Oort shows that $\{W(\Dieu)\pipe 
W\in e\}$ is a filtration by linear subspaces
\begin{equation*}
	0=W_0\subset \cdots\subset W_i\subset\cdots\subset W_r\subset\cdots\subset 
	W_{2r}=\Dieu
\end{equation*}
such that
\begin{enumerate}
	\item For every $0\leq j\leq 2r$ we have $\perp(W_j)=W_{2r-j}$.
	\item There is a surjective function $v:\{0,\ldots,2r\}\to\{0,\ldots,r\}$ 
		such that $F(W_j)=W_{v(j)}$ for every $0\leq j\leq 2r$.
\end{enumerate}
It is called the \emph{canonical filtration of $A$}.
Let $\rho:\{0,\ldots,2r\}\to \mathbb{N}$ be given by $\rk (W_i)=\rho(i)$. We 
associate with $A$ a final sequence $\psi=\psi(A)$ using these data. Suppose 
$\{\psi(0),\psi(1),\ldots,\psi(\rho(i))\}$ has been defined for some $0\leq i$.
Define $\{\psi(0),\ldots,\psi(\rho(i+1))\}$ by 
$\psi(\rho(i))=\psi(\rho(i)+1)=\cdots=\psi(\rho(i+1))$ if $v(i+1)=v(i)$ and by 
$\psi(\rho(i))<\psi(\rho(i)+1)<\cdots<\psi(\rho(i+1))$ if $v(i+1)>v(i)$. We 
denote by $w_A$ the element of $\Wfin$ corresponding to $\psi(A)$.

The main result in this context is

\begin{thm}{\cite[Section 9]{Oort}}\label{EOmain}
	Let $A_1,A_2\in\calA_g(\fk)$. Then $\psi(A_1)=\psi(A_2)$ if and only if 
	$A_1[p]\simeq A_2[p]$ as finite group schemes over $\fk$.
\end{thm}

We will need the following

\begin{lem}\label{DimV2}
	For $A\in\calA_g(\fk)$ we have $\dim\im (V^2)=\psi(g)$.
\end{lem}
\begin{proof}
	It follows from \cite[Remark, p. 18]{Oort} that $\dim\im (F^2)=\psi(g)$.
	Using Proposition \ref{FVandPair}(\ref{imker}) and Corollary \ref{FandV} we 
	see that
	\begin{equation*}
		\im V^2=V(\im V)=V((\im V)^\perp)=(F^{-1}(\im V))^\perp= (F^{-1}(\ker 
		F))^\perp = (\ker F^2)^\perp
	\end{equation*}
	Hence $\dim\im (V^2)=2g-(2g-\dim\im (F^2))=\dim\im (F^2)$. 
\end{proof}

\subsection{The EO stratification}

On $\calA_g$ we have the Ekedahl-Oort stratification (a stratification in the 
sense of Section \ref{prank})
\begin{equation*}
	\mathcal{A}_g=\coprod_{w\in\Wfin}EO_w,
\end{equation*}
given by $A\in EO_w(\fk)$ if and only if $w=w_A$. Using the bijection 
from Section \ref{FinSeq} we will also index the strata by elements of $\ES$.

We list some properties of the EO stratification.
\begin{prop}\label{EOProp}
	Let $w\in \Wfin$.
	\begin{enumerate}
		\item The stratum $EO_w$ is contained in $\calS_g$ if and only if 
			$w(i)=i$ for $1\leq i\leq g-\left[ \frac{g}{2} \right]$.
		\item If 
			$EO_w$ is not contained in $\calS_g$, then $EO_w$ is irreducible.
		\item\label{EOp-rank} The $p$-rank on $EO_w$ is given by \begin{equation*}
				\#\{i\in\{1,\ldots,g\}\pipe w(i)=g+i\}.
			\end{equation*}
		\item The stratum $EO_{\psi}$ is equi-dimensional of dimension
			\begin{equation*}
				\dim EO_{w}=\ell(w)=\sum_{i=1}^g \psi_w(i).
			\end{equation*}
	\end{enumerate}
	Let $\psi,\widetilde{\psi}\in\ES$. 
	\begin{enumerate}
			\setcounter{enumi}{4}
		\item The $a$-number on $EO_{\psi}$ is given by $g-\psi(g)$.
		\item If $\psi(i)\leq \widetilde{\psi}(i)$ for $1\leq i\leq g$ then 
			$EO_{\psi}\subset \overline{EO_{\widetilde{\psi}}}$.
	\end{enumerate}
\end{prop}

\begin{proof}
	(5) can be found in \cite[p. 5]{Harashita}, (6) is shown in 
	\cite[14.3]{Oort}. See \cite[Proposition 2.3 - 
	2.5]{SSLIW} for the other points.
\end{proof}

In view of property (\ref{EOp-rank}) we denote by $\Wfin^{(i)}$  the set of 
final elements of $p$-rank $i$ and by $\ES^{(i)}$ the set of final sequences of $p$-rank $i$,
$0\leq i\leq g$.

\subsection{\texorpdfstring{EO strata for $g=2$}{EO strata for g=2}}\label{EOg2}

The following table contains all final sequences $\psi\in\ES$ and the 
corresponding elements of $\Wfin$ for $g=2$. We also make explicit some of the 
information on $EO_\psi$ contained in Proposition \ref{EOProp}.

\begin{center}
	\begin{tabular}{|c|c|c|c|c|c|}
		\hline
		$\ES$& $\Wfin$& $\dim$& $p$-rank& $a$-number& $\subset\calS_2$?\\
		\hline
		$(0,0)$& $\id$& 0& 0& 2& $\surd$\\
		\hline
		$(0,1)$& $s_2=\begin{pmatrix}1&2&3&4\\1&3&2&4\end{pmatrix}$& 1& 0& 1& 
		$\surd$\\
		\hline
		$(1,1)$& $s_{12}=\begin{pmatrix}1&2&3&4\\2&4&1&3\end{pmatrix}$& 2& 1& 1& 
		$-$\\
		\hline
		$(1,2)$& $s_{212}=\begin{pmatrix}1&2&3&4\\3&4&1&2\end{pmatrix}$& 3& 2& 
		0& $-$\\
		\hline
	\end{tabular}
\end{center}

In particular we see that for $g=2$ the relationship between the 
EO stratification and the supersingular locus is very easy to describe: We have 
$\calS_2=EO_{\id}\cup EO_{s_2}$ and $\calS_2\cap EO_{s_{12}}=\calS_2\cap 
EO_{s_{212}}=\emptyset$. 

\subsection{\texorpdfstring{EO strata for $g=3$}{EO strata for g=3}}\label{EOg3}

The following table contains all final sequences $\psi\in\ES$ and the 
corresponding elements of $\Wfin$ for $g=3$. We also make explicit some of the 
information on $EO_\psi$ contained in Proposition \ref{EOProp}.

\begin{center}
	\begin{tabular}{|c|c|c|c|c|c|}
		\hline
		$\ES$& $\Wfin$& $\dim$& $p$-rank& $a$-number& $\subset\calS_3$?\\
		\hline
		$(0,0,0)$& $\id$& 0& 0& 3& $\surd$\\
		\hline
		$(0,0,1)$& $s_3=\begin{pmatrix}1&2&3&4&5&6\\1&2&4&3&5&6\end{pmatrix}$& 
		1&0& 2& $\surd$\\
		\hline
		$(0,1,1)$& 
		$s_{23}=\begin{pmatrix}1&2&3&4&5&6\\1&3&5&2&4&6\end{pmatrix}$& 2&0& 2& 
		$-$\\
		\hline
		$(0,1,2)$& 
		$s_{323}=\begin{pmatrix}1&2&3&4&5&6\\1&4&5&2&3&6\end{pmatrix}$& 3&0& 1& 
		$-$\\
		\hline
		$(1,1,1)$& 
		$s_{123}=\begin{pmatrix}1&2&3&4&5&6\\2&3&6&1&4&5\end{pmatrix}$& 3& 1& 2& 
		$-$\\
		\hline
		$(1,1,2)$& 
		$s_{3123}=\begin{pmatrix}1&2&3&4&5&6\\2&4&6&1&3&5\end{pmatrix}$& 4& 1& 
		1& $-$\\
		\hline
		$(1,2,2)$& 
		$s_{23123}=\begin{pmatrix}1&2&3&4&5&6\\3&5&6&1&2&4\end{pmatrix}$& 5& 2& 
		1& $-$\\
		\hline
		$(1,2,3)$& 
		$s_{323123}=\begin{pmatrix}1&2&3&4&5&6\\4&5&6&1&2&3\end{pmatrix}$&6&3&0&$-$\\
		\hline
	\end{tabular}
\end{center}

\subsection{\texorpdfstring{The isomorphisms $\Psi_A$}{The isomorphisms Psi\_A}}\label{StaBas}

For $n\in\mathbb{N}$ we endow $\fk^{2n}$ with the non-degenerate alternating 
pairing $\paird=\paird_{\mathrm{def}}$ defined on the standard basis 
$(e_i)_{i=1}^{2n}$ by $\left<e_i,e_j\right>=0=\left<e_{n+i},e_{n+j}\right>$ and 
$\left<e_i,e_{n+j}\right>=\delta_{i,j}$ for $i,j\in\{1,\ldots,n\}$ (note that 
this pairing is different from the one used in Section \ref{GroThe}). 

Let $A\in\calA_g(\fk)$ and $w=w_A$. In 
\cite[Section 9]{Oort}, Oort constructs an isomorphism $\Psi_A:\fk^{2g}\to 
\Dieu(A)$ such that the endomorphisms $F_w=\Psi_A^*F$ and 
$V_w=\Psi_A^*V$ of $\fk^{2g}$ map standard basis vectors to standard basis 
vectors up to sign and such that $\Psi_A^*\paird_A=\paird_{\mathrm{def}}$. As 
the notation indicates these pullbacks only depend on $w$. They are given as 
follows.

Let $w\in W_{\text{final},g}$ with corresponding final sequence $\psi\in\ES$. Denote by $1\leq 
m_1<m_2<\cdots<m_g\leq 2g$ the set of all $m\in\{1,\ldots,2g\}$ with 
$\psi(m-1)<\psi(m)$. Denote by $1\leq n_g<n_{g-1}<\cdots<n_1\leq 2g$ the 
complementary set. In particular $m_i+n_i=2g+1$ for all $1\leq i\leq g$. 
Now for $1\leq i,j\leq g$ we have $F_w(e_{2g+1-i})=0$ and 
\begin{equation*}
	F_w(e_i)=\left\{\begin{matrix} e_j&\text{if }i=m_j\\
		e_{g+j}&\text{if }i=n_j	\end{matrix}\right.
\end{equation*}
Furthermore 
\begin{equation*}
	V_w(e_i)=\left\{\begin{matrix} -e_{g+n_i}&\text{if }n_i\leq g\\
		0&\text{if }m_i\leq g	\end{matrix}\right.
\end{equation*}
and
\begin{equation*}
	V_w(e_{g+i})=\left\{\begin{matrix} e_{g+m_i}&\text{if }m_i\leq g\\
		0&\text{if }n_i\leq g	\end{matrix}\right.
\end{equation*}

As we are going to 
make extensive use of these pullbacks in the cases $g=2$ and $g=3$ we make this 
description explicit in the next subsections.

\subsection{\texorpdfstring{$g=2$}{g=2}}

We describe the pullbacks $F_w$ and $V_w$ depending on $w\in\Wfin^{(0)}$ for $g=2$ 
with respect to the standard basis $(e_1,\ldots,e_4)$.

\begin{center}
	\begin{tabular}{|c|c|c|}
		\hline
		$\Wfin$&$F_w$&$V_w$\\
		\hline
		$\id$ & $\begin{pmatrix}0&0&0&0\\0&0&0&0\\0&1&0&0\\1&0&0&0\end{pmatrix}$ 
		& $\begin{pmatrix}0&0&0&0\\0&0&0&0\\0&-1&0&0\\-1&0&0&0\end{pmatrix}$\\
		\hline
		$s_2$ & $\begin{pmatrix}0&1&0&0\\0&0&0&0\\0&0&0&0\\1&0&0&0\end{pmatrix}$ 
		& $\begin{pmatrix}0&0&0&0\\0&0&0&0\\0&-1&0&0\\0&0&1&0\end{pmatrix}$\\
		\hline
	\end{tabular}
\end{center}

\subsection{\texorpdfstring{$g=3$}{g=3}}

We describe the pullbacks $F_w$ and $V_w$ depending on $w\in\Wfin^{(0)}$ for $g=3$ 
with respect to the standard basis $(e_1,\ldots,e_6)$.

\begin{center}
	\begin{longtable}{|c|c|c|}
		\hline
		$\Wfin$&$F$&$V$\\
		\hline
		$\id$ & 
		$\begin{pmatrix}0&0&0&0&0&0\\0&0&0&0&0&0\\0&0&0&0&0&0\\0&0&1&0&0&0\\0&1&0&0&0&0\\1&0&0&0&0&0\end{pmatrix}$ 
		& 
		$\begin{pmatrix}0&0&0&0&0&0\\0&0&0&0&0&0\\0&0&0&0&0&0\\0&0&-1&0&0&0\\0&-1&0&0&0&0\\-1&0&0&0&0&0\end{pmatrix}$\\
		\hline
		$s_3$ & 
		$\begin{pmatrix}0&0&1&0&0&0\\0&0&0&0&0&0\\0&0&0&0&0&0\\0&0&0&0&0&0\\0&1&0&0&0&0\\1&0&0&0&0&0\end{pmatrix}$ 
		&
		$\begin{pmatrix}0&0&0&0&0&0\\0&0&0&0&0&0\\0&0&0&0&0&0\\0&0&-1&0&0&0\\0&-1&0&0&0&0\\0&0&0&1&0&0\end{pmatrix}$\\
		\hline
		$s_{23}$ & 
		$\begin{pmatrix}0&1&0&0&0&0\\0&0&0&0&0&0\\0&0&0&0&0&0\\0&0&0&0&0&0\\0&0&1&0&0&0\\1&0&0&0&0&0\end{pmatrix}$ 
		& 
		$\begin{pmatrix}0&0&0&0&0&0\\0&0&0&0&0&0\\0&0&0&0&0&0\\0&0&-1&0&0&0\\0&0&0&1&0&0\\0&-1&0&0&0&0\end{pmatrix}$\\
		\hline
		$s_{323}$ & 
		$\begin{pmatrix}0&1&0&0&0&0\\0&0&1&0&0&0\\0&0&0&0&0&0\\0&0&0&0&0&0\\0&0&0&0&0&0\\1&0&0&0&0&0\end{pmatrix}$ 
		& 
		$\begin{pmatrix}0&0&0&0&0&0\\0&0&0&0&0&0\\0&0&0&0&0&0\\0&0&-1&0&0&0\\0&0&0&1&0&0\\0&0&0&0&1&0\end{pmatrix}$\\
		\hline
	\end{longtable}
\end{center}

\section{\texorpdfstring{The EO stratification and $\calS_g$ for $g=3$}{The EO 
stratification and S\_g for g=3}}\label{EOandSS3}

We determine the relationship between the EO stratification and the 
supersingular locus $\calS_g$ for $g=3$. As $\calS_g\subset \calA^{(0)}$ we only 
have to look at the EO strata of $p$-rank 0.

\begin{thm}\label{EOandSSthm}
	Let $g=3$.
	\begin{enumerate}
		\item\label{continSS} $EO_{\id},EO_{s_3}\subset \calS_3$.
		\item\label{SSands23} $EO_{s_{23}}\cap \calS_3=\emptyset$.
		\item\label{SSands323} $EO_{s_{323}}\cap\calS_3$ is a dense open subset of $\calS_3$ of 
			pure dimension 2.
	\end{enumerate}
\end{thm}

\begin{proof}
	\begin{enumerate}
		\item See Proposition \ref{EOProp} or Section \ref{EOg3}. 
		\item By Section \ref{EOg3} the 
			$a$-number on $EO_{s_{23}}$ is equal to $2$. Let $A\in\calS_3(\fk)$ 
			be a supersingular abelian variety with $a(A)=2$ and choose a basis 
			$(\overline{X}_1,\overline{X}_2,\overline{X}_3,\allowbreak\overline{Y}_1,\overline{Y}_2,\overline{Y}_3)$ 
			of $\Dieu(A[p])$, a matrix $\overline{T}$ and an element 
			$\overline{\varepsilon}$ with the properties of Section 
			\ref{SSDieMod}. We have $\rk(\overline{T})=g-a(A)=1$ and the 
			symmetry condition for $\overline{T}$ then implies that there is a 
			$t\in\fk^\times$ with 
			$\overline{T}=\begin{pmatrix}0&0&0\\0&0&0\\t&0&0\end{pmatrix}$. We 
			deduce that the canonical filtration of $A$ is given by 
			\begin{align*}
				0&\subset \left<\overline{Y}_1\right>\subset 
				\left<\overline{Y}_1,\overline{Y}_2\right>\subset 
				\left<t\overline{X}_3+\overline{\varepsilon}\overline{Y}_3,\overline{Y}_1,\overline{Y}_2\right>\subset 
				\left<\overline{Y}_1,\overline{Y}_2,\overline{Y}_3,\overline{X}_3\right>\\
				&\subset 
				\left<\overline{Y}_1,\overline{Y}_2,\overline{Y}_3,\overline{X}_2,\overline{X}_3\right>\subset\Dieu(A)
			\end{align*}
			with $r=g=3,\ \rho(i)=i$ for $0\leq i\leq 6$ and $v(0)=v(1)=v(2)=0,\ 
			v(3)=v(4)=1,\ v(5)=2,\ v(6)=3$. Hence $A$ lies in $EO_{s_3}$ and our 
			claim is shown.
		\item Set $U=EO_{s_{323}}\cap\calS_3$. By Proposition \ref{EOProp} we 
			know that $\overline{EO_{s_{323}}}=\calA_3^{(0)}$. As $EO_{s_{323}}$ 
			is locally closed, this implies that $EO_{s_{323}}$ is open in 
			$\calA_3^{(0)}$, hence $U$ is open in $\calS_3$. Now $\calS_3$ is 
			equi-dimensional of dimension 2, hence the same is true for every 
			nonempty open subset of $\calS_3$. But $\calS_3- U=EO_{\id}\cup 
			EO_{s_3}=\overline{EO_{s_3}}$ has dimension 1 and this implies that 
			$U$ intersects every irreducible component of $\calS_3$, hence it is 
			even dense in $\calS_3$.
	\end{enumerate}
\end{proof}

\begin{remark}
	According to \cite[Remark 1, p. 8]{Harashita} it is true for any $g\geq 3$ 
	that $EO_{(0,1,\ldots,g-1)}\cap \calS_g$ is open and dense in $\calS_g$.
\end{remark}

\section{Flag varieties and corresponding notation}\label{flagnot}

We want to study the fibers of $\pi:\calA_I\to\calA_g$. Instead of investigating 
them directly we will look at their image under an injective, finite morphism 
with values in a suitable flag variety. This is sufficient if we are only 
interested in their topological properties. Before introducing this morphism in 
the next Section we have to fix some notation concerning flag varieties.

\subsection{Flag varieties}

Let $n\in\mathbb{N}$. For $0\leq i\leq n$ we denote by $\Flag_{i,n}$ the variety 
of partial flags $(W_j)_{j=0}^i$ in $\fk^n$ satisfying $\dim W_j=j$ for all $0\leq 
j\leq i$. We write 
$\Flag_n=\Flag_{n,n}$. We denote by $\OFlag_{2n}$ the variety of full symplectic 
flags in $\fk^{2n}$ with respect to the pairing $\paird=\paird_{\mathrm{def}}$ 
defined in Section \ref{StaBas}. If we have an element $(W_i)_{i=0}^n$ of 
$\Flag_{n,2n}(\fk)$ with $W_n$ totally isotropic we will occasionally consider 
it as an element of $\OFlag_{2n}(\fk)$ by implicitly extending it to the flag 
$(W_i)_{i=0}^{2n}$ with $W_{2n-i}=W_i^\perp$ for $0\leq i\leq n$. 

Let $g\geq 1$ and $w\in\Wfin$. We denote by $\Flag^{F,V}_w=\Flag^{F,V}_{w,2g}$ the closed 
subvariety of $\Flag_{2g}$ whose $\fk$-valued points are given by those 
flags $(W_i)_{i=0}^{2g}$ satisfying 
\begin{equation}
	F_w(W_i),V_w(W_i)\subset W_i\tag{$*$}
\end{equation}
for all $0\leq i\leq 2g$.
We write $\AFlag_{w}=\AFlag_{w,2g}=\Flag_{w,2g}^{F,V}\cap \OFlag_{w,2g}$. It 
follows from Proposition \ref{FVandPair}(\ref{pair}) that an element 
$(W_i)_{i=0}^{2g}\in \OFlag_{2g}(\fk)$ lies in $\AFlag_w(\fk)$ if and only if 
it satisfies condition $(*)$ for $0\leq i\leq g$.

\subsection{Standard charts}\label{StaCha}

For a field $K$ and $k,l\in\mathbb{N}$ we denote by $\mathrm{FM}^{k\times l}(K)$ the set 
of $k\times l$-matrices with entries in $K$ of full rank. Let $n,i$ be as 
above.  There is a canonical surjection $\overline{\Phi}:\mathrm{FM}^{n\times 
i}(\fk)\to\Flag_{i,n}(\fk)$ sending a matrix $B$ to the flag $(W_j)_{j=0}^i$ with 
$W_j$ spanned by the first $j$ columns of $B$.

Given pairwise distinct $j_1,j_2,\ldots,j_i\in\{1,2,\ldots,n\}$ we denote by 
$U_{j_1,\ldots,j_i}$ the open subset of $\Flag_{i,n}$ whose $\fk$-valued points 
are given by the image under $\overline{\Phi}$ of the set of matrices 
$C=(c_{kl})\in \mathrm{FM}^{n\times i}(\fk)$ satisfying
\begin{enumerate}
	\item $c_{j_ll}=1$ for all $l\in\{1,\ldots,i\}$,
	\item $c_{j_ll'}=0$ for all $l\in\{1,\ldots,i\}$ and every 
		$l'\in\{l+1,\ldots,i\}$.
\end{enumerate}
An open subset of this form will be called a \emph{standard chart} for 
$\Flag_{i,n}$. Considered as open subschemes of $\Flag_{i,n}$ we identify them 
with an affine space of an appropriate dimension in the usual way. Obviously we 
have
\begin{equation*}
	\Flag_{i,n}=\bigcup_{\substack{j_1,j_2,\ldots,j_i\in\{1,2,\ldots,n\}\\ 
	\text{pairwise distinct}}}U_{j_1,\ldots,j_i}
\end{equation*}
Hence in order to prove that some subset of $\Flag_{i,n}$ is closed we will show that its 
intersection with all standard charts is closed.  Furthermore morphisms into and 
out of $\Flag_{i,n}$ will be obtained by glueing morphisms into and out of 
standard charts respectively.

\subsection{\texorpdfstring{Subvarieties of $\OFlag$}{Subvarieties of the 
variety of symplectic flags}}\label{uglynotation}

Consider the set
\begin{equation*}
	\mathrm{FM}^{2n,\perp}(\fk)=\left\{ B\in \mathrm{FM}^{2n\times n}(\fk) \left|
	B^t \begin{pmatrix}0&-1_n\\1_n&0\end{pmatrix}B=0\right.\right\}.
\end{equation*}
Then $\overline{\Phi}$ restricts to a surjection
$\Phi:\mathrm{FM}^{2n,\perp}(\fk)\to\OFlag_{2n}(\fk)$.

We need an economic notation for defining subvarieties of $\OFlag_{2n}$. We 
think that such a notation is most easily explained via an example. Consider a 
table of the following form.
\begin{center}
	\begin{tabular}{|c|c|c|}
		\hline
		$Z_1$&$Z_2$&$Z_3$ \\
		\hline
		$\begin{pmatrix}
			0&x\\
			1&0\\
			0&y\\
			0&0
		\end{pmatrix}$&
		$\begin{pmatrix}
			0&0\\
			0&x\\
			1&0\\
			0&y
		\end{pmatrix}$&
		$\begin{pmatrix}
			0\\
			\GL{2}(\fk)
		\end{pmatrix}\vee
		\begin{pmatrix}
			0&a\\
			1&0\\
			0&b\\
			0&0
		\end{pmatrix}$\\
		\hline 
\multicolumn{2}{|c|}{$(x,y)^t\in(\FF{p})^2-\{0\}$}&
		$(a,b)^t\in\fk^2-\{0\}$\\
		\hline
		\hline
		\multicolumn{2}{|c|}{$\coprod_{p+1}\Spec 
		\fk$}&$\PR{1}_{\fk}\coprod\PR{1}_\fk$\\
		\hline
		\multicolumn{2}{|c|}{0}&1\\
		\hline
	\end{tabular}
\end{center}

The upper block defines subvarieties $Z_1$, $Z_2$ and $Z_3$ of 
$\OFlag_4$ whose $\fk$-valued points are given by
\begin{equation*}
	Z_1(\fk)=\Phi\left( \left\{ B\in \mathrm{FM}^{4\times 2,\perp}(\fk)\left|
	\exists (x,y)^t\in(\FF{p})^2-\{0\}\text{ s.t. }
	B=\begin{pmatrix}
		0&x\\
		1&0\\
		0&y\\
		0&0
	\end{pmatrix}
	\right.\right\} \right),
\end{equation*}
\begin{equation*}
	Z_2(\fk)=\Phi\left( \left\{ B\in \mathrm{FM}^{4\times 2,\perp}(\fk)\left|
	\exists (x,y)^t\in(\FF{p})^2-\{0\}\text{ s.t. }
	B=\begin{pmatrix}
		0&0\\
		0&x\\
		1&0\\
		0&y
	\end{pmatrix}
	\right.\right\} \right),
\end{equation*}
and
\begin{equation*}
	Z_3(\fk)=\Phi\left( \left\{ B\in\mathrm{FM}^{4\times 2,\perp}(\fk)\left|
	\begin{matrix}\exists 
	A\in\GL{2}(\fk)\text{ s.t. } B=\begin{pmatrix}0\\A\end{pmatrix}
	\vee\\
	\exists (a,b)^t\in\fk^2-\{0\}\text{ s.t. }
	B=\begin{pmatrix}
		0&a\\
		1&0\\
		0&b\\
		0&0
	\end{pmatrix}\end{matrix}
	\right.\right\} \right).
\end{equation*}

Furthermore the lower block of the table claims that there are isomorphisms $Z_1\simeq 
\coprod_{p+1}\Spec \fk$, $Z_2\simeq \coprod_{p+1}\Spec \fk$ 
and $Z_3\simeq \PR{1}_\fk\coprod \PR{1}_\fk$ and that $\dim Z_1=0$, $\dim Z_2=0$ and 
$\dim Z_3=1$. Note our convention for rows spanning multiple columns where the 
contained information is to be applied to every column separately. The notation 
is not meant to imply any connection between the individual subvarieties.

Note that this notation is highly ambiguous. For instance we could have written 
$Z_1$ ineptly as
\begin{center}
	\begin{tabular}{|c|}
		\hline
		$Z_1$\\
		\hline
		$\begin{pmatrix}
			0&a\\
			1&\alpha\\
			0&b\\
			0&0
		\end{pmatrix}$\\
		\hline
		$(a,b)^t\in\fk^2-\{0\},\ \alpha\in\fk$\\$a^pb-ab^p=0$\\
		\hline
	\end{tabular}
\end{center}

\section{\texorpdfstring{The maps $\iota_A$}{The maps iota\_A}}

\subsection{de Rham cohomology}

Let $g\geq 1$. Let $f:A\to S$ be an abelian scheme of relative dimension $g$. We 
denote by $\Omega_{A/S}^\bullet$ the de Rham complex of $\calO_A$-modules. The 
\emph{first de Rham cohomology sheaf $H^1_{DR}(A/S)$} is defined by 
\begin{equation*}
	H^1_{DR}(A/S)=R^1f_*(\Omega^\bullet).
\end{equation*}
It is a locally free $\calO_S$-module of rank $2g$, functorial in $A$, and its 
formation commutes with base-change by \cite[Proposition 2.5.2]{BBM}.

Inside $H^1_{DR}(A/S)$ we have the \emph{Hodge filtration} $\omega_A$, a locally 
free $\calO_S$-submodule of rank $g$, given by the image of the injection 
$R^0f_*(\Omega^1_{A/S})\to H^1_{DR}(A/S)$ coming from the Hodge-de Rham spectral 
sequence, cf. \cite[Proposition 2.5.3]{BBM}.

Let $A/\fk$ be an abelian variety. In \cite{Oda}, Oda constructs a natural 
isomorphism $\Dieu(A)\overset{\sim}{\to} H^1_{DR}(A/\fk)$ taking $\im V$ to 
$\omega_A$. See in particular \cite[Cor. 5.11]{Oda}.

\subsection{\texorpdfstring{The map $\iota$}{The map iota}}
We recall a construction from \cite[Section 4]{SSLIW}. Let 
$f:A^{\mathrm{univ}}\to\calA_g$ be the universal abelian scheme and consider its 
de Rham cohomology 
$\mathbb{H}=H^1_{DR}(A^{\mathrm{univ}}/\calA_g)$. Denote by 
$\Flag(\mathbb{H})\to\calA_g$ the variety of full flags in $\mathbb{H}$. We 
define a morphism $\calA_I\xrightarrow{\iota}\Flag(\mathbb{H})$ over $\calA_g$ 
on $S$-valued points as follows: Let $(A_i)_{i=0}^g\in \calA_I(S)$. We extend it 
to a chain $(A_i)_{i=0}^{2g}$ by setting $A_{2g-i}=A_i^\vee$ for $0\leq i<g$. The map 
$A_{2g-i}\to A_{2g-i+1}$ is given by the dual isogeny of $A_{g-i-1}\to A_{g-i}$ 
for $0\leq i<g$, while the map $A_g\to A_{g+1}$ is given by the 
composition $A_g\too{\lambda_g}A_g^\vee\too{\alpha^\vee}A_{g-1}^\vee=A_{g+1}$.

Then the image of $(A_i)_i$ in $\Flag(\mathbb{H})(S)$ is given by
\begin{equation*}
	0=\alpha(H^1_{DR}(A_{2g}))\subset 
	\alpha(H^1_{DR}(A_{2g-1}))\subset\cdots\subset \alpha(H^1_{DR}(A_1))\subset
	H^1_{DR}(A_0),
\end{equation*}
where for each $i$, $\alpha$ denotes the
map induced by $A_0\to A_i$. The morphism $\iota$ is universally
injective and finite, see \cite[Lemma 4.3]{SSLIW}. 

\begin{defn}
	Let $A\in\calA_g(\fk)$ and consider the final element $w_A\in\Wfin$.
	We denote by $\iota_A$ the composition
	\begin{equation*}
		\pi^{-1}(A)\to \Flag(H^1_{DR}(A/\fk))\too{\sim} \Flag(\Dieu(A))\too{\sim}
		\Flag_{2g}
	\end{equation*}
	where the first map is obtained from $\iota$ by base-change, the second map is 
	induced by Oda's isomorphism mentioned above and the third map is induced 
	by the isomorphism $\Psi_A$ from Section \ref{StaBas}. 
\end{defn}

Let $A\in\calA_g(\fk)$ and $w=w_A$. It follows from classical Dieudonn\'e theory 
that the image of $\iota_A$ is given by $\AFlag_{w,2g}$. Hence $\iota_A$ induces 
a universal homeomorphism $\pi^{-1}(A)\to\AFlag_{w,2g}$. If we are only 
interested in topological properties of the fiber $\pi^{-1}(A)$, it is therefore 
sufficient to study the spaces $\AFlag_{w,2g}$. The following sections 
contain a list of the varieties $\AFlag_{w,2g}$ for $g=2$ and $g=3$. 
\section{\texorpdfstring{The varieties $\AFlag_{w,4}$ over the $p$-rank 0 
locus}{The varieties Flag\_{w,4} over the p-rank 0 locus}}\label{Fib2}

Let $g=2$. Depending on $w\in\Wfin^{(0)}$ we determine the variety
$\AFlag_w\subset\OFlag_4$ by writing down its irreducible components. We use 
the notation explained in Section \ref{uglynotation}.

\subsection{\texorpdfstring{$w=\id$}{w=id}}
Let $J=\{x\in \fk\pipe x^p=-x\}$. The irreducible components of $\AFlag_{\id}$ 
are given by $Y$, $(Z_x)_{x\in J}$ and $Z_{\infty}$ with
\begin{center}
	\begin{tabular}{|c|c|c|}
		\hline
		$Z$ & $Z_x$ & $Z_{\infty}$\\
		\hline
		$\begin{pmatrix}
			0\\
			\GL{2}(\fk)
		\end{pmatrix}$ & $\begin{pmatrix}
			0&a\\
			0&-xa\\
			x&b\\
			1&0
		\end{pmatrix}$ & $\begin{pmatrix}
			0&0\\
			0&a\\
			1&0\\
			0&b
		\end{pmatrix}$\\
		\hline & \multicolumn{2}{c|}{$(a,b)^t\in\fk^2-\{0\}$}\\
		\hline
		\hline
		\multicolumn{3}{|c|}{$\PR{1}_{\fk}$}\\
		\hline
		\multicolumn{3}{|c|}{1}\\
		\hline
	\end{tabular}
\end{center}
The $Z_\zeta$ are pairwise disjoint and each $Z_\zeta$ intersects $Z$ in precisely one 
point, as $\zeta$ runs through $J\cup \{\infty\}$. Hence there are $p+2$ irreducible components and 
$\AFlag_{\id}$ is connected. 

\subsection{\texorpdfstring{$w=s_2$}{w=s\_2}}

$\AFlag_{s_2}$ is given by \begin{center}
	\begin{tabular}{|c|}
		\hline
		$\begin{pmatrix}
			0&a\\
			0&0\\
			0&b\\
			1&0
		\end{pmatrix}$\\
		\hline $(a,b)^t\in\fk^2-\{0\}$\\
		\hline
		\hline
		$\PR{1}_{\fk}$\\
		\hline
		1\\
		\hline
	\end{tabular}
\end{center}

\section{\texorpdfstring{The varieties $\AFlag_{w,6}$ over the $p$-rank 0 
locus}{The varieties Flag\_{w,6} over the p-rank 0 locus}}\label{Fib3}

Let $g=3$. Depending on $w\in\Wfin^{(0)}$ we determine the varieties 
$\AFlag_w\subset\OFlag_6$. We use the notation introduced in Section 
\ref{uglynotation}. For a matrix $B\in M^{3\times 2}(\fk)$ we denote by $B_i$ 
the matrix obtained from $B$ by deleting the $i$-th row, $i=1,2,3$. Furthermore 
we denote by $\mathrm{B}_{*}(\PR{2}_{\fk})$ the blowing-up of $\PR{2}_{\fk}$ in a 
closed point.

\subsection{\texorpdfstring{$w=\id$}{w=id}}\label{FibI}

Let $I=\{(x,y)^t\in (\FF{p^2})^2\pipe x^p+x+y^{p+1}=0\}$.  The 
irreducible components of $\AFlag_{\id}$ are given by $Y$, $Z$, $T_{\infty}$ and 
$(T_{x,y})_{(x,y)^t\in I}$:
\begin{center}
	\begin{longtable}{|c|c|}
		\hline
		Y&Z\\
		\hline
		$\begin{pmatrix}
			0\\
			\GL{3}(\fk)
		\end{pmatrix}$&
		$\begin{pmatrix}
			0&\det{B_1}\\
			0&-\det{B_2}\\
			0&\det{B_3}\\
			B&\fk^3
		\end{pmatrix}\,\vee\,
		\begin{pmatrix}
			0\\
			\begin{matrix}
				B&v
			\end{matrix}
		\end{pmatrix}$\\
		\hline
		&$B\in\mathrm{FM}^{3\times 2}(\fk),\ v\in\fk^3$\\
		&$\begin{pmatrix}B&v\end{pmatrix}\in\GL{3}(\fk)$\\
		&$\det{B_1^p}\det{B_3}+\det{B_2^{p+1}}+\det{B_1}\det{B_3^p}=0$\\
		\hline
		\hline
		$\Flag_{3}(\fk)$&\\
		\hline
		\multicolumn{2}{|c|}{3}\\
		\hline
		\multicolumn{2}{c}{}\\
		\newpage
		\hline
		$T_\infty$&	$T_{x,y}$\\
		\hline
		$\begin{pmatrix}
			0&0&0\\
			0&0&0\\
			0&a&0\\
			1&0&0\\
			0&b&1\\
			0&c&0
		\end{pmatrix}\,\vee\,
		\begin{pmatrix}
			0&0&0\\
			0&0&0\\
			0&0&\alpha\\
			1&0&0\\
			0&1&0\\
			0&0&\beta
		\end{pmatrix}$&
		$\begin{pmatrix}
			0&a&0\\
			0&ya&0\\
			0&xa&0\\
			x^p&b&-y\\
			y^p&c&1\\
			1&0&0
		\end{pmatrix}\,\vee\,
		\begin{pmatrix}
			0&0&\alpha\\
			0&0&y\alpha\\
			0&0&x\alpha\\
			x^p&-y&\beta\\
			y^p&1&0\\
			1&0&0
		\end{pmatrix}$\\
		\hline
		$\begin{matrix}(a,b,c)^t\in\fk^3-\fk\cdot(0,1,0)^t\\(\alpha,\beta)^t\in\fk^2-\{0\}\end{matrix}$&
		$\begin{matrix} (a,b,c)^t\in\fk^3-\fk\cdot(0,-y,1)^t\\(\alpha,\beta)^t\in\fk^2-\{0\} \end{matrix}$\\
		\hline
		\hline
		$\mathrm{B}_{*}(\PR{2}_{\fk})$&$\mathrm{B}_{*}(\PR{2}_{\fk})$\\
		\hline
		2&2\\
		\hline
	\end{longtable}
\end{center}
In order to get a better understanding of $Z$ we look at the closed subvariety 
$Z_0$ of $\Flag_{2,3}$ whose $\fk$-valued points are the image under 
$\overline{\Phi}$ of the set
$\{B\in \mathrm{FM}^{3\times 2}(\fk)\pipe 
\det{B_1^p}\det{B_3}+\det{B_2^{p+1}}+\det{B_1}\det{B_3^p}=0\}$. 
There is an obvious surjective morphism 
$\gamma:Z\to Z_0$ and $Z$ becomes a $\PR{1}_{\fk}$-bundle over $Z_0$ via 
$\gamma$. It is trivial over the intersection of $Z_0$ with any of the standard 
charts of $\Flag_{2,3}$.

Now $\Flag_{2,3}$ is itself a $\PR{1}_{\fk}$-bundle over $\mathrm{Grass}_{2,3}$, the 
variety of 2-dimensional subspaces of $\fk^3$, and if we identify
$\mathrm{Grass}_{2,3}$ with $\PR{2}_{\fk}$ in the usual way, the map 
$\delta:\Flag_{2,3}(\fk)\to\PR{2}_{\fk}(\fk)$ of this bundle is given by 
$\delta(\overline{\Phi}(B))=(\det B_1:\det B_2:\det B_3)$, where 
$B\in\mathrm{FM}^{3\times 2}$. Choosing homogenous coordinates $X_1$, $X_2$ and 
$X_3$ on $\PR{2}_{\fk}$, $\delta$ restricts to a map
$\varepsilon:Z_0\to V_+(X_1^pX_3+X_2^{p+1}+X_1X_3^p)$ making $Z_0$ a 
$\PR{1}_{\fk}$-bundle over the curve $V_+(X_1^pX_3+X_2^{p+1}+X_1X_3^p)\subset 
\PR{2}_{\fk}$.

\begin{equation*}
	Z\xrightarrow{\PR{1}_{\fk}\text{-bundle}}Z_0\xrightarrow{\PR{1}_{\fk}\text{-bundle}}V_+(X_1^pX_3+X_2^{p+1}+X_1X_3^p)
\end{equation*}

As it is not immediately obvious from the table above what the intersection 
between the individual irreducible components are, we list them separately. 
First of all the $T_\zeta$ are pairwise disjoint, as $\zeta$ runs through $I\cup 
\{\infty\}$. 
\begin{center}
	\begin{longtable}{|c|c|c|c|}
		\hline
		\multicolumn{2}{|c|}{$Y\cap Z$}& $Y\cap T_\infty$&$Y\cap T_{x,y}$\\
		\hline
		\multicolumn{2}{|c|}{$\begin{pmatrix}
			0\\
			\begin{matrix}
				B&v
			\end{matrix}
		\end{pmatrix}$}&
		$\begin{pmatrix}
			0\\
			\begin{matrix}
				1&0\\
				0&\GL{2}(\fk)
			\end{matrix}
		\end{pmatrix}$&
		$\begin{pmatrix}
			0\\
			\begin{matrix}
				\begin{matrix}
					x^p\\
					y^p
				\end{matrix}&
				\GL{2}(\fk)\\
				1&0
			\end{matrix}
		\end{pmatrix}$\\
		\hline
		\multicolumn{2}{|c|}{$\begin{matrix}
		B\in\mathrm{FM}^{3\times 2}(\fk),\ v\in\fk^3\\
		\begin{pmatrix}B&v\end{pmatrix}\in\GL{3}(\fk)\\
		\det{B_1^p}\det{B_3}+\det{B_2^{p+1}}+\\
		\det{B_1}\det{B_3^p}=0\end{matrix}$}&&\\
		\hline
		\hline
		\multicolumn{2}{|c|}{$Z_0$}&\multicolumn{2}{c|}{$\PR{1}_{\fk}$}\\
		\hline
		\multicolumn{2}{|c|}{2}&\multicolumn{2}{c|}{1}\\
		\hline
		\multicolumn{4}{c}{}\\
		\newpage
		\hline
		$Z\cap T_\infty$&$Z\cap T_{x,y}$&$Y\cap Z\cap T_\infty$&$Y\cap Z\cap T_{x,y}$\\
		\hline
		$\begin{pmatrix}
			0&0&0\\
			0&0&0\\
			0&0&\alpha\\
			1&0&0\\
			0&1&0\\
			0&0&\beta
		\end{pmatrix}$&
		$\begin{pmatrix}
			0&0&\alpha\\
			0&0&y\alpha\\
			0&0&x\alpha\\
			x^p&-y&\beta\\
			y^p&1&0\\
			1&0&0
		\end{pmatrix}$&
		$\begin{pmatrix}
			0\\
			\begin{matrix}
				1&0&0\\
				0&1&0\\
				0&0&1	
			\end{matrix}
		\end{pmatrix}$&
		$\begin{pmatrix}
			0\\
			\begin{matrix}
				x^p&-y&1\\
				y^p&1&0\\
				1&0&0
			\end{matrix}
		\end{pmatrix}$\\
		\hline
		\hline
		\multicolumn{2}{|c|}{$(\alpha,\beta)^t\in\fk^2-\{0\}$}&&\\
		\hline
		\multicolumn{2}{|c|}{$\PR{1}_\fk$}&\multicolumn{2}{c|}{$\Spec \fk$}\\
		\hline
		\multicolumn{2}{|c|}{1}&\multicolumn{2}{c|}{0}\\
		\hline
	\end{longtable}
\end{center}

\begin{remark}
	In \cite[Section 2]{Richartz} Richartz studies a similar situation in order 
	to investigate the geometry of the supersingular locus $\calS_3$. Let us 
	briefly explain how her situation is related to ours. Denote by 
	$(\fk^6,F,V,\paird)=(\fk^6,F_\id,V_\id,\paird_{\mathrm{def}})$ the 
	superspecial Dieudonn\'e module in dimension 3. We look at the variety $X$ 
	of flags $(W_2\subset W_3\subset W_4)$ in $\fk^6$ with $W_3^\perp=W_3$ and 
	$\dim W_i=i,\ F(W_i),V(W_i)\subset W_i$ for all $i\in\{2,3,4\}$. Inside $X$ 
	we have the subvariety $\tilde{X}$ given by those $(W_2\subset W_3\subset 
	W_4)\in X(\fk)$ satisfying $F(W_4),V(W_4)\subset W_2$ and $\im F\subset 
	M_4$, compare \cite[2.8]{Richartz}.

	Richartz shows that the image of $\tilde{X}$ under the canonical projection 
	$X\to\mathrm{Grass}_{2,6}$ is isomorphic to the curve 
	$C=V_+(X_1^{p+1}+X_2^{p+1}+X_3^{p+1})\subset\PR{2}_\fk$ and that
	the restriction $\tilde{X}\to C$ has fibers isomorphic to $\PR{1}_\fk$. 
	Furthermore she shows that the restriction of the canonical projection 
	$X\to\mathrm{Grass}_{3,6}$ to $\tilde{X}$ is birational onto its image. This 
	image is denoted by $\calL_0(N)$ in loc.cit. It is the
	key tool used in \cite{Richartz} to describe the structure of $\calS_3$.

	To relate these objects to our situation, look at the morphism 
	$\zeta:\AFlag_{\id,6}\to X$ given by $(W_i)_{i=0}^6\to (W_i)_{i=2}^4$ for 
	$(W_i)_{i=0}^6\in \AFlag_{\id,6}(\fk)$. Consider a point
	$(W_i)_{i=2}^4\in X(\fk)$ and the corresponding endomorphisms
	$F_{|W_2}$ and $V_{|W_2}$ of $W_2$ induced by $F$ and $V$, respectively and 
	write $\calU=\ker F_{|W_2}\cap\ker V_{|W_2}$. Then we see as in the proof of 
	Proposition \ref{ConnectedFib} below that the fiber of $\zeta$ over $(W_i)_{i=2}^4$ is 
	nonempty and isomorphic to $\mathbb{P}(\calU)$. It is isomorphic to $\PR{1}_\fk$ if 
	and only if $\calU=W_2$ and consists of one point else. The first case occurs if 
	and only if $W_2\subset \ker F\cap\ker V$, which is equivalent to $\im 
	F\subset W_4$.

	It is easily checked that $\zeta^{-1}(\tilde{X})=Z$ and that the fibers of the 
	restriction $\tilde{\zeta}:Z\to\tilde{X}$ of $\zeta$ are isomorphic to 
	$\PR{1}_\fk$. Hence we get the following picture:
	\begin{equation*}
		\xymatrix{
		\AFlag_{\id,6}\ar[d]\ar@{}[r]|{\supseteq}&Z\ar[d]^{\PR{1}_\fk}&\\
		X\ar@{}[r]|{\supseteq}&\tilde{X}\ar[d]^{\PR{1}_\fk}\ar[r]^-{\mathrm{bir.}}&\calL_0(N)\\
		&C&
		}
	\end{equation*}
\end{remark}

\subsection{\texorpdfstring{$w=s_3$}{w=s\_3}}\label{FibII}

The irreducible components of $\AFlag_{s_3}$ are given by \begin{center}
	\begin{tabular}{|c|c|}
		\hline
		$X$&$Y$\\
		\hline
		$\begin{pmatrix}
			0&0&\alpha\\
			0&0&0\\
			0&0&0\\
			0&0&\beta\\
			a&b&0\\
			c&d&0
		\end{pmatrix}$ &
		$\begin{pmatrix}
			0&a&0\\
			0&0&0\\
			0&0&0\\
			0&b&0\\
			0&c&1\\
			1&0&0
		\end{pmatrix}\,\vee\,
		\begin{pmatrix}
			0&0&\alpha\\
			0&0&0\\
			0&0&0\\
			0&0&\beta\\
			0&1&0\\
			1&0&0
		\end{pmatrix}$\\
		\hline
		$\begin{pmatrix}a&b\\c&d\end{pmatrix}\in\GL{2}(\fk)$ & 
		$(a,b,c)^t\in\fk^3-\fk\cdot(0,0,1)^t$\\
		$(\alpha,\beta)^t\in\fk^2-\{0\}\quad\quad$&$(\alpha,\beta)^t\in\fk^2-\{0\}\quad\quad$\\
		\hline
		\hline
		$\PR{1}_{\fk}\times\PR{1}_{\fk}$&$\mathrm{B}_{*}(\PR{2}_{\fk})$\\
		\hline
		2&2\\
		\hline
	\end{tabular}
\end{center}
Hence $\AFlag_{s_3}$ consists of two planes intersecting in the exceptional 
curve of $Y$.

\subsection{\texorpdfstring{$w=s_{23}$}{w=s\_23}}\label{FibIII}

The irreducible components of $\AFlag_{s_{23}}$ are given by
\begin{center}
	\begin{tabular}{|c|c|c|c|c|}
		\hline
		$X$&$Y_1$&$Y_2$&$Z_1$&$Z_2$\\
		\hline
		$\begin{pmatrix}
			0&0&\alpha\\
			0&0&0\\
			0&0&0\\
			0&0&\beta\\
			a&b&0\\
			c&d&0
		\end{pmatrix}$&
		$\begin{pmatrix}
			0&0&0\\
			0&0&0\\
			0&0&0\\
			0&a&b\\
			1&0&0\\
			0&c&d
		\end{pmatrix}$ &
		$\begin{pmatrix}
			0&a&b\\
			0&0&0\\
			0&0&0\\
			0&0&0\\
			0&c&d\\
			1&0&0
		\end{pmatrix}$ &
		$\begin{pmatrix}
			0&1&0\\
			0&0&\alpha\\
			0&0&0\\
			0&0&0\\
			0&0&\beta\\
			1&0&0
		\end{pmatrix}$ &
		$\begin{pmatrix}
			0&0&0\\
			0&0&0\\
			0&0&\alpha\\
			0&1&0\\
			1&0&0\\
			0&0&\beta
		\end{pmatrix}$\\
		\hline
		$\begin{matrix}\begin{pmatrix}a&b\\c&d\end{pmatrix}\in\GL{2}(\fk)\\
		(\alpha,\beta)^t\in\fk^2-\{0\}\end{matrix}$&\multicolumn{2}{c|}{$\begin{pmatrix}a&b\\
		c&d\end{pmatrix}\in\GL{2}(\fk)$}&
		\multicolumn{2}{c|}{$(\alpha,\beta)^t\in\fk^2-\{0\}$}\\
		\hline
		\hline
		$\PR{1}_{\fk}\times\PR{1}_{\fk}$&\multicolumn{4}{c|}{$\PR{1}_{\fk}$}\\
		\hline
		2&\multicolumn{4}{c|}{1}\\
		\hline	
	\end{tabular}
\end{center}
We have $Y_1\cap Y_2=Z_1\cap Z_2=X\cap Z_1=X\cap Z_2=\emptyset$. The curves $Y_1$ and $Y_2$ 
each intersect the plane $X$ in precisely one point. The intersections 
$Y_1\cap Z_1$ and $Y_2\cap Z_2$ also consist of precisely one point each.

\subsection{\texorpdfstring{$w=s_{323}$}{w=s\_323}}\label{FibIV}

The irreducible components of $\AFlag_{s_{323}}$ are given by

\begin{center}
	\begin{longtable}{|c|c|c|}
		\hline
		$X$&$Y$&$Z$\\
		\hline
		$\begin{pmatrix}
			0&1&0\\
			0&0&\alpha\\
			0&0&0\\
			0&0&0\\
			0&0&\beta\\
			1&0&0
		\end{pmatrix}$&
		$\begin{pmatrix}
			0&0&\alpha\\
			0&0&0\\
			0&0&0\\
			0&0&\beta\\
			0&1&0\\
			1&0&0
		\end{pmatrix}$&
		$\begin{pmatrix}
			0&a&b\\
			0&0&0\\
			0&0&0\\
			0&0&0\\
			0&c&d\\
			1&0&0
		\end{pmatrix}$\\
		\hline
		\multicolumn{2}{|c|}{$(\alpha,\beta)^t\in\fk^2-\{0\}$}&$
		\begin{pmatrix}
			a&b\\
			c&d
		\end{pmatrix}\in \GL{2}(\fk)$\\
		\hline
		\hline
		\multicolumn{3}{|c|}{$\PR{1}_{\fk}$}\\
		\hline
		\multicolumn{3}{|c|}{1}\\
		\hline
	\end{longtable}
\end{center}
We have $X\cap Y=\emptyset$ while the intersections $X\cap Z$ and $Y\cap Z$ 
consist of precisely one point each.

\section{Proof of the results of Sections \ref{Fib2} and 
\ref{Fib3}}\label{proofflag}

The case $w=\id$ for $g=3$ is obviously the most complicated one and we use it 
to illustrate the method. Assume that 
\begin{equation*}
	C=(c_1c_2c_3)=
	\begin{pmatrix}
		c_{11}&c_{12}&c_{13}\\
		c_{21}&c_{22}&c_{23}\\
		c_{31}&c_{32}&c_{33}\\
		c_{41}&c_{42}&c_{43}\\
		c_{51}&c_{52}&c_{53}\\
		c_{61}&c_{62}&c_{63}
	\end{pmatrix}\in\mathrm{FM}^{6\times 3,\perp}(\fk)
\end{equation*}
is such that $(W_i)_i=\Phi(C)\in\AFlag_{\id}(\fk)$. We evaluate the condition that 
every step of the flag $\Phi(C)$ is stable under $F_\id$ and $V_\id$, where we use 
the explicit description of $F_\id$ and $V_\id$ given in Section \ref{StaBas}. As we 
are only interested in the image of $C$ under $\Phi$ we may without loss of 
generality multiply columns of $C$ by elements of $\fk^*$ and add 
$\fk^*$-multiples of $c_i$ to $c_j$ for $1\leq i<j\leq 3$. We will do so below 
without further mentioning.

$W_1$ is stable under $F$ and $V$ if and only if 
\begin{equation*}
	\begin{pmatrix}0\\0\\0\\c_{31}^p\\c_{21}^p\\c_{11}^p\end{pmatrix},\begin{pmatrix}0\\0\\0\\-c_{31}^{p^{-1}}\\-c_{21}^{p^{-1}}\\-c_{11}^{p^{-1}}\end{pmatrix}\in 
	\fk\cdot 
	\begin{pmatrix}c_{11}\\c_{21}\\c_{31}\\c_{41}\\c_{51}\\c_{61}\end{pmatrix}
\end{equation*}
This is satisfied if and only if $c_{11}=c_{21}=c_{31}=0$. 

For $W_2$ the conditions are trivially satisfied if $(c_{12},c_{22},c_{32})=0$.
If this vector is not zero, we consider several cases.
\begin{enumerate}
	\item $c_{61}\neq 0$: We may assume that $(c_1c_2)$ is of the form
		\begin{equation*}
			\begin{pmatrix}
				0&c_{12}\\
				0&c_{22}\\
				0&c_{32}\\
				c_{41}&c_{42}\\
				c_{51}&c_{52}\\
				1&0
			\end{pmatrix}
		\end{equation*}
		$W_2$ is stable under $F$ and $V$ if and only if \begin{equation*}
			\begin{pmatrix}c_{32}^p\\c_{22}^p\\c_{12}^p\end{pmatrix},\begin{pmatrix}-c_{32}^{p^{-1}}\\-c_{22}^{p^{-1}}\\-c_{12}^{p^{-1}}\end{pmatrix}\in 
			\fk\cdot\begin{pmatrix}c_{41}\\c_{51}\\1\end{pmatrix}
		\end{equation*}
		This is the case if and only if $c_{41},c_{51}\in\FF{p^2}$ and 
		$c_{32}=c_{12}c_{41}^p,\ c_{22}=c_{12}c_{51}^p$. Hence we see that 
		we may assume that $(c_1c_2)$ is of the form 
		\begin{equation*}
			\begin{pmatrix}
				0&1\\
				0&y\\
				0&x\\
				x^p&b\\
				y^p&c\\
				1&0
			\end{pmatrix}\			\end{equation*}
		for some $x,y\in\FF{p^2}$ and some $b,c\in\fk$. The fact that $C$ is 
		supposed to be an element of $\mathrm{FM}^\perp(\fk)$ implies that we have 
		$x^p+y^{p+1}+x=0$.
	\item $c_{61}=0,\ c_{51}\neq 0$: We may assume that $(c_1c_2)$ is of the 
		form
		\begin{equation*}
			\begin{pmatrix}
				0&c_{12}\\
				0&c_{22}\\
				0&c_{32}\\
				c_{41}&c_{42}\\
				1&0\\
				0&c_{62}
			\end{pmatrix}
		\end{equation*}
		The stability of $W_2$ under $F$ implies that \begin{equation*}
			\begin{pmatrix}c_{32}^p\\c_{22}^p\\c_{12}^p\end{pmatrix}\in\fk\cdot\begin{pmatrix}c_{41}\\1\\0\end{pmatrix}
		\end{equation*}
		From this we get that $c_{12}=0$ and $c_{22}\neq 0$, which is impossible 
		as $C\in\mathrm{FM}^\perp(\fk)$ implies that $c_{22}+c_{12}c_{41}=0$.
	\item $c_{61}=c_{51}=0$: We may assume that $(c_1c_2)$ is of the form
		\begin{equation*}
			\begin{pmatrix}
				0&c_{12}\\
				0&c_{22}\\
				0&c_{32}\\
				1&0\\
				0&c_{52}\\
				0&c_{62}
			\end{pmatrix}
		\end{equation*}
		$C\in\mathrm{FM}^\perp(\fk)$ implies $c_{12}=0$. $W_2$ is stable under $F$ and $V$ 
		if and only if
		\begin{equation*}
			\begin{pmatrix}c_{32}^p\\c_{22}^p\\0\end{pmatrix},\begin{pmatrix}-c_{32}^{p^{-1}}\\-c_{22}^{p^{-1}}\\0\end{pmatrix}\in 
			\fk\begin{pmatrix}1\\0\\0\end{pmatrix}
		\end{equation*}
		This implies that $c_{22}=0$ and we see that $(c_1c_2)$ can be chosen of 
		the form
		\begin{equation*}
			\begin{pmatrix}
				0&0\\
				0&0\\
				0&1\\
				1&0\\
				0&b\\
				0&c
			\end{pmatrix}
		\end{equation*}
		for some $b,c\in\fk$.
\end{enumerate}

For $W_3$ we first assume that $(c_{12},c_{22},c_{32})=0$. We write 
\begin{equation*}
	B=(b_1b_2)=\begin{pmatrix}c_{41}&c_{42}\\c_{51}&c_{52}\\c_{61}&c_{62}\end{pmatrix}.
\end{equation*}
If $(c_{13},c_{23},c_{33})=0$ the conditions are trivially satisfied. The flags 
of this form are contained in the set $Y$. If this vector is not zero,
$C\in\mathrm{FM}^\perp(\fk)$ implies $(c_{13},c_{23},c_{33})^t\in (\fk\cdot b_1\oplus 
\fk\cdot b_2)^{\perp_{can}}$, 
where $\perp_{can}$ refers to the canonical pairing $(x,y)\mapsto x^ty$ on 
$\fk^3$. But $(\fk\cdot b_1\oplus \fk\cdot b_2)^{\perp_{can}}$ is spanned by the vector $(\det B_1, -\det B_2, \det 
B_3)^t$ and hence we may assume that \begin{equation*}
	(c_{13},c_{23},c_{33})=(\det B_1,-\det B_2,\det B_3).
\end{equation*}
The stability of $W_3$ under $F$ and $V$ is equivalent to
\begin{equation*}
	\begin{pmatrix}\det B_1^p\\-\det B_2^p\\\det 
		B_3^p\end{pmatrix},\begin{pmatrix}\det B_1^{p^{-1}}\\-\det 
		B_2^{p^{-1}}\\\det B_3^{p^{-1}}\end{pmatrix}\in \fk\cdot b_1+\fk\cdot 
		b_2
\end{equation*}
This can also be expressed as the vanishing of the determinants of the matrices
\begin{equation*}
	M_1=
	\begin{pmatrix}
		B&\begin{matrix}
			\det B_1^p\\
			-\det B_2^p\\
			\det B_3^p
		\end{matrix}
	\end{pmatrix}
	\quad\text{and}\quad
	M_2=
	\begin{pmatrix}
		B&\begin{matrix}
			\det B_1^{p^{-1}}\\
			-\det B_2^{p^{-1}}\\
			\det B_3^{p^{-1}}
		\end{matrix}
	\end{pmatrix}
\end{equation*}
But we have $\det M_1=\det M_2^p$ and hence we are left with the equation $\det 
M_1=0$, which is equal to the equation 
$\det{B_1^p}\det{B_3}+\det{B_2^{p+1}}+\det{B_1}\det{B_3^p}=0$. Hence we see that 
the flags of this form are contained in the set $Z$.

Finally we have to consider the case where $(c_{12},c_{22},c_{32})\neq 0$. We do 
this accordingly to the cases introduced in the discussion of $W_2$ above.
\begin{enumerate}
	\item We may assume that $C$ is of the form
		\begin{equation*}
			\begin{pmatrix}
				0&1&0\\
				0&y&c_{23}\\
				0&x&c_{33}\\
				x^p&b&c_{43}\\
				y^p&c&c_{53}\\
				1&0&0
			\end{pmatrix}			
		\end{equation*}
		for some $x,y\in\FF{p^2}$ with $x^p+y^{p+1}+x=0$ and some $b,c\in\fk$. We see that the 
		stability of $W_3$ under $F$ and $V$ implies that $(c_{23},c_{33})=0$.
		$C\in\mathrm{FM}^\perp(\fk)$ then implies that $c_{43}=-yc_{53}$ and we see that 
		$c_{53}\neq 0$. Hence we may assume that $C$ is of the form
		\begin{equation*}
			\begin{pmatrix}
				0&1&0\\
				0&y&0\\
				0&x&0\\
				x^p&b&-y\\
				y^p&c&1\\
				1&0&0
			\end{pmatrix}			
		\end{equation*}
		and we see that flags of this form are contained in the set $T_{x,y}$.
\setcounter{enumi}{2}
	\item We may assume that $C$ is of the form
		\begin{equation*}
			\begin{pmatrix}
				0&0&c_{13}\\
				0&0&c_{23}\\
				0&1&0\\
				1&0&0\\
				0&b&c_{53}\\
				0&c&c_{63}
			\end{pmatrix}
		\end{equation*}
		for some $b,c\in\fk$.
		The stability of $W_3$ under $F$ and $V$ implies $c_{13}=c_{23}=0$.
		$c_3\perp c_2$ implies $c_{63}=0$. Hence $C$ is of the form 
		\begin{equation*}
			\begin{pmatrix}
				0&0&0\\
				0&0&0\\
				0&1&0\\
				1&0&0\\
				0&b&1\\
				0&c&0
			\end{pmatrix}
		\end{equation*}
		and flags of this type are contained in the set $T_\infty$.
\end{enumerate}

Conversely it is easily checked that the sets $Y$, $Z$, $T_{\infty}$ and 
$(T_{x,y})_{(x,y)^t\in I}$ defined above are indeed subsets of $\AFlag(\fk)$. In order 
to show that they are closed and to construct the isomorphisms claimed above one 
has to calculate their intersection with the standard charts. 

For instance we have
\begin{center}
	\begin{tabular}{|c|c|}
		\hline
		$Z\cap U_{453}$& $Z\cap U_{456}$\\
		\hline
		$\begin{pmatrix}
			0&0&b_{21}b_{32}-b_{31}\\
			0&0&-b_{32}\\
			0&0&1\\
			1&0&0\\
			b_{21}&1&0\\
			b_{31}&b_{32}&\alpha
		\end{pmatrix}$&
		$\begin{pmatrix}
			0&0&\alpha(b_{21}b_{32}-b_{31})\\
			0&0&-\alpha b_{32}\\
			0&0&\alpha\\
			1&0&0\\
			b_{21}&1&0\\
			b_{31}&b_{32}&1
		\end{pmatrix}$\\
		\hline
		\multicolumn{2}{|c|}{$b_{21},b_{31},b_{32},\alpha\in\fk$}\\
		\multicolumn{2}{|c|}{$(b_{21}b_{32}-b_{31})^p+b_{32}^{p+1}+(b_{21}b_{32}-b_{31})=0$}\\
		\hline
	\end{tabular}
\end{center}
These are closed subsets of the respective affine spaces. 

To see that $\gamma:Z\to Z_0$ is a $\PR{1}_{\fk}$-bundle, trivial over the 
intersections of $Z_0$ with the standard charts of $\Flag_{2,3}$, we note 
exemplarily that the preimage of $V_{45}=Z_0\cap U_{45}$ under $\gamma:Z\to Z_0$ 
is given by $(Z\cap U_{453})\cup (Z\cap U_{456})$. It is now easy to
define an isomorphism $\gamma^{-1}(V_{45})\to V_{45}\times \PR{1}_{\fk}$ fitting 
into the commutative diagram 
\begin{equation*}
	\xymatrix{
	\gamma^{-1}(V_{45})\ar[dr]_{\gamma\pipe_{\gamma^{-1}(V_{45})}}\ar@{.>}[rr]^{\simeq}&&V_{45}\times\PR{1}_\fk\ar[dl]^{\text{pr}}\\
	&V_{45}&
	}.
\end{equation*}
Here $V_{45}$ is the hypersurface in $\mathbb{A}^3_\fk$ given by the image under 
$\overline{\Phi}$ of the set of matrices $B\in\mathrm{FM}^{3\times 2}(\fk)$ of the 
form $B=\begin{pmatrix}1&0\\b_{21}&1\\b_{31}&b_{32}\end{pmatrix}$ for 
$b_{21},b_{31},b_{32}\in\fk$ with 
$(b_{21}b_{32}-b_{31})^p+b_{32}^{p+1}+(b_{21}b_{32}-b_{31})=0$.

To prove that $Z$ is irreducible of dimension $3$ it suffices to show that 
$V_+(X_1^pX_3+X_2^{p+1}+X_1X_3^p)$ is irreducible (see Proposition \ref{DimProp} 
below). If we consider $X_1^pX_3+X_2^{p+1}+X_1X_3^p$ as an element of 
$K[X_1,X_3][X_2]$, we can apply Eisenstein's criterion (using the prime 
element $X_1$ of $K[X_1,X_3]$) to see that $X_1^pX_3+X_2^{p+1}+X_1X_3^p$ is 
irreducible.

Concerning the intersections only the statement about $Z\cap T_{x,y}$ is not 
quite obvious, namely that it is contained in the exceptional 
curve of $T_{x,y}$. Let $(a,b,c)^t\in\fk^3-\fk\cdot(0,-y,1)^t$ and assume that 
\begin{equation*}
	\Phi \begin{pmatrix}
			0&a&0\\
			0&ya&0\\
			0&xa&0\\
			x^p&b&-y\\
			y^p&c&1\\
			1&0&0
		\end{pmatrix}\in Z.
\end{equation*}
First this implies $a=0$. If $c=0$ the condition on the determinants for 
elements in $Z$ would imply $b=0$, but $(a,b,c)\neq 0$ by assumption. Hence 
$c\neq 0$ and we may assume $c=1$. Then the determinant condition becomes 
$-(x^p-by^p)+(-b)^{p+1}-(x^p-by^p)^p=0$. Using $x^p+x+y^{p+1}=0$ we get 
$y^{p(p+1)}+by^p+b^{p+1}+b^py^{p^2}=0$. Using $y\in\FF{p^2}$ this becomes 
$y^{p+1}+by^p+b^{p+1}+b^py=0$ which is equivalent to $(y+b)^{p+1}=0$. As $b\neq 
-y$ by assumption this does not have a solution.

\section{\texorpdfstring{The case of positive $p$-rank}{The case of positive p-rank}}\label{secposprank}

In order to investigate the fiber of $\pi$ over abelian varieties of positive 
$p$-rank we need to recall some additional material concerning finite 
commutative group schemes over $\fk$ and their Dieudonn\'e theory. Our main 
reference is again \cite{Dem}.

Let $g\geq 1$ and $A\in\calA_g(\fk)$.
Then $A[p]$ is in a unique way a product of three subgroups $A[p]=G^{e,u}\times 
G^{i,m}\times G^{i,u}$ with $G^{e,u}$ \'etale unipotent, $G^{i,m}$ infinitesimal multiplicative and $G^{i,u}$ 
infinitesimal unipotent. One has isomorphisms $G^{e,u}\simeq 
(\mathbb{Z}/p\mathbb{Z})^k$ and $G^{i,m}\simeq\mu_p^k$, where $k$ is equal to 
the $p$-rank of $A$. Here $\mu_p$ denotes the $\fk$-group scheme representing 
the functor $S\mapsto \{s\in\calO_S(S)\pipe s^p=1\}$ on the category of 
$\fk$-schemes. In terms of Dieudonn\'e modules this corresponds to a 
decomposition $\Dieu=W^{e,u}\oplus W^{i,m}\oplus W^{i,u}$ into subspaces stable 
under $F$ and $V$ and such that 
\begin{itemize}
	\item $F_{|W^{e,u}}$ is an isomorphism and $V_{|W^{e,u}}$ is nilpotent,
	\item $F_{|W^{i,m}}$ is nilpotent and $V_{|W^{i,m}}$ is an isomorphism,
	\item $F_{|W^{i,u}}$ and $V_{|W^{i,u}}$ are nilpotent.
\end{itemize}
Here we write $F_{|W^{e,u}}$ for the morphism $W^{e,u}\to W^{e,u}$ induced by 
$F$ etc. We have $\dim_\fk W^{e,u}=\dim_\fk W^{i,m}=k$ and $\dim_\fk 
W^{i,u}=2(g-k)$.

This decomposition is natural: If $\widetilde{A}\in\calA_g(\fk)$ with decomposition 
$\Dieu(\widetilde{A}[p])=\widetilde{W}^{e,u}\oplus \widetilde{W}^{i,m}\oplus 
\widetilde{W}^{i,u}$ and if $\alpha:A[p]\to \widetilde{A}[p]$ is a group 
homomorphism, the induced morphism $\Dieu(\alpha):\Dieu(\widetilde{A}[p])\to 
\Dieu(A[p])$ splits into the direct sum of three morphisms 
$\widetilde{W}^{e,u}\to W^{e,u}$, $\widetilde{W}^{i,m}\to W^{i,m}$ and 
$\widetilde{W}^{i,u}\to W^{i,u}$. In particular 
\begin{equation}\label{natsplit}
	\begin{split}
	\im \Dieu(\alpha)
	&= \im \Dieu(\alpha)\cap (W^{e,u}\oplus W^{i,m})\oplus \im \Dieu(\alpha)\cap 
	W^{i,u} \\
	&=\im \Dieu(\alpha)\cap W^{e,u}\oplus \im \Dieu(\alpha)\cap W^{i,m}\oplus 
	\im \Dieu(\alpha)\cap W^{i,u}.
\end{split}
\end{equation}

\begin{lem}\label{posprank}
	Let $k\geq 0$ and $A\in\calA^{(k)}_g(\fk)$ with decomposition $\Dieu=W^{e,u}\oplus W^{i,m}\oplus 
	W^{i,u}$ as above. Let $\fk^{2g}=U^{e,u}\oplus U^{i,m}\oplus 
	U^{i,u}$ be the decomposition induced via $\Psi_A$. Let $w=w_A$ and consider 
	the associated data $\psi$ and $(m_i)_{i=1}^g$ introduced in Section 
	\ref{StaBas}. Write $I=\{i\in\{1,\ldots,g\}\pipe i=m_i\}$ and 
	$I^c=\{1,\ldots,g\}-I$. 
	\begin{enumerate}
		\item We have	
			\begin{align*}
				U^{e,u}&=\bigoplus_{i\in I}\fk\cdot e_i\\
				U^{i,m}&=\bigoplus_{i\in I}\fk\cdot e_{g+i}\\
				U^{i,u}&=\bigoplus_{i\in I^c}(\fk\cdot e_i\oplus \fk\cdot e_{g+i})
			\end{align*}
\item\label{second} Let $\widetilde{w}\in W_{\mathrm{final},g-k}$ be the final element 
			corresponding to the final sequence 
			$[\psi(k+1)-k,\psi(k+2)-k,\ldots,\psi(g)-k]$. Denote by 
			$(\widetilde{e}_i)_{i=1}^{2(g-k)}$ the standard basis and by 
			$\widetilde{\paird}=\widetilde{\paird_{\mathrm{def}}}$ the pairing 
			introduced in Section \ref{StaBas}
			on $\fk^{2(g-k)}$.
			Consider the morphism $\widetilde{\beta}:\fk^{2(g-k)}\to U^{i,u}$ given by 
			$\widetilde{e}_i\mapsto e_{k+i}$ and $\widetilde{e}_{g-k+i}\mapsto 
			e_{g+k+i}$ for $1\leq i\leq g-k$. Then $\widetilde{\beta}$ induces an isomorphism of 
			quadruples
			\begin{equation*}
				(\fk^{2(g-k)},F_{\widetilde{w}}, 
				V_{\widetilde{w}},\widetilde{\paird})\to (U^{i,u}, 
				F_{w|U^{i,u}},U_{w|U^{i,u}},\paird_{|U^{i,u}}).
			\end{equation*}
		\item Let $U=U^{e,u}\oplus U^{i,m}$. Let $\widehat{w}\in W_{\mathrm{final},g-k}$ be the final element 
			corresponding to the final sequence 
			$[1,2,\ldots,k]$. Denote by 
			$(\widehat{e}_i)_{i=1}^{2k}$ the standard basis and by 
			$\widehat{\paird}=\widehat{\paird_{\mathrm{def}}}$ the pairing 
			introduced in Section \ref{StaBas} on $\fk^{2k}$.
			Consider the morphism $\widehat{\beta}:\fk^{2k}\to U$ given by 
			$\widehat{e}_i\mapsto e_{i}$ and $\widehat{e}_{k+i}\mapsto 
			e_{g+i}$ for $1\leq i\leq k$. Then $\widehat{\beta}$ induces an isomorphism of 
			quadruples
			\begin{equation*}
				(\fk^{2k},F_{\widehat{w}}, 
				V_{\widehat{w}},\widehat{\paird})\to (U, 
				F_{w|U},U_{w|U},\paird_{|U}).
			\end{equation*}
	\end{enumerate}
\end{lem}

\begin{proof}
	This is an easy consequence of the explicit descriptions of $F_w$ and $V_w$ 
	given in Section \ref{StaBas}.
\end{proof}

\begin{cor}
	Let $g\geq 1$ and $w\in W_{\mathrm{final},g}$. Then the $p$-rank on $EO_w$ is given by 
	$w(1)-1$.
\end{cor}

\begin{proof}
	As above, let $\psi_w$ and $(m_i)_{i=1}^g$ be the data associated with $w$. We see from Lemma 
	\ref{posprank} that the $p$-rank on $EO_w$ is given by 
	$\#\{i\in\{1,\ldots,g\}\pipe i=m_i\}=\#\{i\in\{1,\ldots,g\}\pipe 
	\psi_w(i)=i\}$. It follows from Section \ref{FinSeq} that for all 
	$i\in\{1,\ldots,g\}$ we have $\psi_w(i)=i\Leftrightarrow (\forall 
	a\in\{1,\ldots, g\}\ w(a)>i)$. As $w$ is final, this is equivalent to 
	$w(1)>i$ and there are $w(1)-1$ elements of $\{1,\ldots,g\}$ satisfying this 
	inequality.
\end{proof}

\begin{remark}
	In this way we have obtained a formula for the $p$-rank on an EO stratum 
	which is considerably simpler than the one cited in Proposition 
	\ref{EOProp}(\ref{EOp-rank}). Given both formulas it is of course easy to 
	show by combinatorial means that they are equivalent.
\end{remark}

\begin{prop}\label{splitflag}Let $K$ be an algebraically closed field and $\calV$ a vector space 
	of finite dimension $g$ over $K$. Let $\calV_1$ and $\calV_2$ be subspaces of $\calV$ with 
	$\calV=\calV_1\oplus \calV_2$ and $\dim_K \calV_1=k,\quad \dim_K \calV_2=g-k$. Fix integers 
	$0\leq i\leq g$, $\max(0,i+k-g)\leq l\leq \min(i,k)$ and a subset $J\subset 
	\{1,\ldots,i\}=:I$ of cardinality $l$. 
	\begin{enumerate}
		\item Consider the set
			\begin{equation*}
				Z_{g,i,k,J}=\left\{ (\calF_j)_{j=0}^i\in \Flag_{i,g}(\calV)\pipe 
				\forall j\in I\left (
				\begin{matrix}
					\calF_j=\calF_j\cap \calV_1\oplus \calF_j\cap 
					\calV_2\mathrel{\wedge} \\ \left(  \calF_j\cap 
					\calV_1\neq \calF_{j-1}\cap \calV_1\Leftrightarrow j\in J\right)
				\end{matrix}\right) \right\},
			\end{equation*}
			where $\Flag_{i,g}(\calV)$ denotes the (classical) variety of flags 
			$(\calF_j)_{j=0}^i$ in $\calV$ with $\dim \calF_j=j$ for $0\leq j\leq i$. Then 
			$Z_{g,i,k,J}$ is a closed subvariety of $\Flag_{i,g}(\calV)$. 
			
			Consider the maps $\varphi:\{0,1,\ldots,i\}\to \{0,1,\ldots,l\}$ and 
			$\varphi':\{0,1,\ldots,i\}\to \{0,1,\ldots,g-l\}$  with 
			$\varphi(0)=0=\varphi'(0)$,
			\begin{equation*}
				\varphi(j)=\left\{\begin{matrix}\varphi(j-1),\quad j\notin 
					J\\\varphi(j-1)+1,\quad j\in J\end{matrix}\right.
			\end{equation*}
			and
			\begin{equation*}
				\varphi'(j)=\left\{\begin{matrix}\varphi'(j-1),\quad j\in 
					J\\\varphi'(j-1)+1,\quad j\notin J\end{matrix}\right. .
			\end{equation*}
			Then the map $\alpha_J:\Flag_{l,k}(\calV_1)\times \Flag_{i-l,g-k}(\calV_2)\to 
			Z_{g,i,k,J}$ given by \begin{equation*}
			((\calF_j)_{j=0}^l,(\calG_j)_{j=0}^{i-l})\mapsto 
			(\calF_{\varphi(j)}\oplus \calG_{\varphi'(j)})_{j=0}^i 
		\end{equation*}
			is an isomorphism of (classical) varieties.
		\item Consider the set
			\begin{equation*}
				Z_{g,i,k,l}=\left\{ (\calF_j)_{j=0}^i\in \Flag_{i,g}(\calV)\pipe 
				\begin{matrix}
					\dim_K(\calF_i\cap \calV_1)=l\quad\quad\wedge\\
					(\forall j\in I\ \calF_j=\calF_j\cap \calV_1\oplus \calF_j\cap 
					\calV_2 )
				\end{matrix} \right\}.
			\end{equation*}
			Then $Z_{g,i,k,l}$ is a closed subvariety of $\Flag_{i,g}(\calV)$ and
			\begin{equation*}
				Z_{g,i,k,l}=\coprod_{\substack{J\subset I\\ \# J = l}} 
				Z_{g,i,k,J}\simeq 
				\coprod_{\begin{pmatrix}i\\l\end{pmatrix}}\Flag_{l,k}(\calV_1)\times 
				\Flag_{i-l,g-k}(\calV_2).
			\end{equation*}
		\item Consider the set 	\begin{equation*}
				Z_{g,i,k}=\left\{ (\calF_j)_{j=0}^i\in \Flag_{i,g}(\calV)\pipe 
				\forall j\in I\ 
				\calF_j=\calF_j\cap \calV_1\oplus \calF_j\cap \calV_2 \right\}.
			\end{equation*}
			Then $Z_{g,i,k}$ is a closed subvariety of $\Flag_{i,g}(\calV)$ and 
			\begin{equation*}
				Z_{g,i,k}=\coprod_{l=\max(0,i+k-g)}^{\min(i,k)} Z_{g,i,k,l}\simeq 
				\coprod_{l=\max(0,i+k-g)}^{\min(i,k)}\coprod_{\begin{pmatrix}i\\l\end{pmatrix}}\Flag_{l,k}(\calV_1)\times 
				\Flag_{i-l,g-k}(\calV_2).
			\end{equation*}
	\end{enumerate}
\end{prop}

Consider integers $0\leq i\leq n$. The Frobenius $\sigma:\fk\to \fk$ induces an 
automorphism $\Sigma:\Flag_{i,n}(\fk)\to \Flag_{i,n}(\fk)$. It is given as 
follows: Denote by $\rho:\fk^n\to \fk^n$ the componentwise application of 
$\sigma$. Then for $(\calF_j)_{j=0}^i\in\Flag_{i,n}(\fk)$ we have $\Sigma( 
(\calF_j)_{j=0}^i)=(\rho(\calF_j))_{j=0}^i$. We denote by $\Flag_{i,n}(\FF{p})$ 
the fixed point set of $\Sigma$. It is a finite subset of $\Flag_{i,n}(\fk)$ and 
it can be identified canonically with the set of flags $(\calF_j)_{j=0}^i$ in 
$(\FF{p})^n$ with $\dim \calF_j=j$ for all $0\leq j\leq i$.

\begin{prop}\label{posto0red}
	Let $g\geq 1$, $k\geq 0$ and $A\in\calA_g^{(k)}(\fk)$ with $w=w_A$.
	\begin{enumerate}
 		\item Assume that $k=g$ (i.e. $A$ is ordinary). Then the fiber of $\pi$ over $A$  is 
			discrete and 
			\begin{equation*}
				\# \left(\pi^{-1}(A)\right)=ON_g:=
				2^g\#\Flag_g(\FF{p})=2^g\prod_{l=0}^{g-1}\sum_{i=0}^l 
				p^i=2^g\frac{\prod_{l=1}^g (p^l-1)}{(p-1)^g}.
			\end{equation*}
		\item Assume that $1\leq k\leq g-1$. Then $\AFlag_w$ is isomorphic to 
			$\begin{pmatrix}g\\k\end{pmatrix}ON_{k}$ disjoint copies of 
			$\AFlag_{\widetilde{w},2(g-k)}$, where $\widetilde{w}$ is as in point 
			(\ref{second}) of Lemma \ref{posprank}. Note that the $p$-rank on $EO_{\widetilde{w}}$ is 
			equal to $0$.
	\end{enumerate}
\end{prop}

\begin{proof}
	We use the notation of Lemma \ref{posprank}.
	\begin{enumerate}
		\item If $k=g$ we have $U^{i,u}=0,\ U^{e,u}=\oplus_{i=1}^g \fk\cdot e_i$ 
			and $U^{i,m}=\oplus_{i=g+1}^{2g} \fk\cdot e_i$. Let $w=w_A$. We use 
			the notation of Proposition \ref{splitflag} for $\calV=\fk^{2g},\ 
			\calV_1=U^{e,u}$ and $\calV_2=U^{i,m}$. By Equation \ref{natsplit} we see 
			that \begin{equation*}
				\AFlag_{w,2g}\subset Z_{2g,g,g}=\coprod_{l=0}^g
				\coprod_{\substack{J\subset I\\ \# J = l}} 
				\alpha_J\left(\Flag_{l,g}(\calV_1)\times 
				\Flag_{g-l,g}(\calV_2)\right).
			\end{equation*}
			Let $0\leq l\leq g$ and $J\subset I$ of cardinality $l$. 

			First note that $F_{|\calV_1}$ is equal to the componentwise application of 
			$\sigma$ (with respect to the basis $(e_1,\ldots,e_g)$) and that 
			$V_{|\calV_1}=0$. On the other hand $V_{|\calV_2}$ is equal to the componentwise application of 
			$\sigma^{-1}$ (with respect to the basis $(e_{g+1},\ldots,e_{2g})$) and 
			$F_{|\calV_2}=0$. From this it follows immediately that 
			$\Flag_w^{F,V}\cap\alpha_J\left(\Flag_{l,g}(\calV_1)\times 
			\Flag_{g-l,g}(\calV_2)\right)\simeq\Flag_{l,g}(\FF{p})\times 
			\Flag_{g-l,g}(\FF{p})$.
			
			Let $\left((\calF_j)_{j=0}^l,(\calG_j)_{j=0}^{g-l}\right)\in 
			\Flag_{l,g}(\calV_1)\times \Flag_{g-l,g}(\calV_2)$. Then the image 
			$\alpha_J\left((\calF_j)_{j=0}^l,(\calG_j)_{j=0}^{g-l}\right)$ is symplectic if and only 
			if $(\calG)_{j=0}^{g-l}$ is actually a flag in the $g-l$ dimensional 
			space $\calV_2\cap \calF_l^\perp$, where we consider $\calF_l$ as 
			a subspace of $\calV$.

			Combining these two statements it follows that 
			\begin{equation*}
				\AFlag_w\cap\alpha_J\left(\Flag_{l,g}(\calV_1)\times 
			\Flag_{g-l,g}(\calV_2)\right)\simeq\Flag_{l,g}(\FF{p})\times 
			\Flag_{g-l}(\FF{p}).
			\end{equation*}
			The claim now follows from a short calculation.
		\item We use the notation of Proposition \ref{splitflag} with 
			$\calV=\fk^{2g},\ \calV_1=U^{e,u}\oplus U^{i,m}$ and $\calV_2=U^{i,u}$. It 
			follows from $\calV_1=\calV_2^\perp$ and Equation \ref{natsplit} that 
			\begin{equation*}
			\AFlag_w\subset Z_{2g,g,2k,k}=\coprod_{\substack{J\subset I\\ \# J = 
			k}}
				\alpha_J\left(\Flag_{k,2k}(\calV_1)\times 
				\Flag_{g-k,2(g-k)}(\calV_2)\right).
			\end{equation*}	
			Let $J\subset I$ of cardinality $k$. Using the notation and the 
			isomorphisms $\widetilde{\beta}$ and $\widehat{\beta}$ of Lemma \ref{posprank} we see 
			that 
			\begin{equation*}
				\AFlag_{w}\cap\alpha_J
			\left(\Flag_{k,2k}(\calV_1)\times \Flag_{g-k,2(g-k)}(\calV_2)\right)\simeq 
			\AFlag_{\widehat{w},2k}\times \AFlag_{\widetilde{w},2(g-k)}.
			\end{equation*}
			By the first point $\AFlag_{\widehat{w},2k}$ is discrete of 
			cardinality $ON_{k}$ and the claim follows.
	\end{enumerate}
\end{proof}

\section{\texorpdfstring{The number of connected components of the fibers of 
$\pi$}{The number of connected components of the fibers of 
pi}}

\begin{prop}\label{ConnectedFib}
	Let $g\geq 1$. For all $A\in\calA_g^{(0)}(\fk)$, 
	the fiber $\pi^{-1}(A)$ is connected.
\end{prop}

\begin{remark}
	In \cite[Proposition 5.2]{Yu2} Yu proves a more general statement. We rephrase his proof in 
	our language.
\end{remark}

\begin{proof}[Proof of Proposition \ref{ConnectedFib}]
	Let $w=w_A$. We have to show that $\AFlag_w$ is connected. Let 
	$0\leq i\leq g,\ I=\{0,\ldots,i\}$ and denote by $\AFlag_{w,i}$ the variety whose $\fk$-valued 
	points are given by
	\begin{equation*}
		\AFlag_{w,i}(\fk)=\left\{(W_j)_{j=0}^i\in\Flag_{i,2g}(\fk)\pipe
		\begin{matrix}
			\forall j\in I\ V_w(W_j),F_w(W_j)\subset W_j\ \\
			\text{ and } W_i\text{ is isotropic}
		\end{matrix}
		\right\}.
	\end{equation*}
	Then $\AFlag_{w,g}=\AFlag_w$ and we will show by induction on $i$ that 
	$\AFlag_{w,i}$ is connected for all $0\leq i\leq g$. For each $1\leq i\leq 
	g$ consider the morphism $\zeta_i:\AFlag_{w,i}\to \AFlag_{w,i-1}$ given by 
	$\zeta_i\left( (W_j)_{j=0}^i \right)=(W_j)_{j=0}^{i-1}$ for 
	$(W_j)_{j=0}^i\in \AFlag_{w,i}(\fk)$. This is, in particular, a closed map of 
	topological spaces and it will be sufficient to show that it is surjective 
	with connected fibers. Fix a point $(W_j)_{j=0}^{i-1}\in 
	\AFlag_{w,i-1}(\fk)$ and write $\calW=W^\perp_{i-1}/W_{i-1}$ with canonical 
	projection $pr:W^\perp_{i-1}\to \calW$. $F_w$ and $V_w$ induce 
	endomorphisms $\overline{F}$ and $\overline{V}$ of $\calW$. Our assumption 
	on the $p$-rank of $A$ implies that $\overline{F}$ and $\overline{V}$ are 
	nilpotent.
	This means that a 1-dimensional subspace of $\calW$ is stable under 
	$\overline{F}$ or $\overline{V}$ if and only if it is contained in $\ker 
	\overline{F}$ or $\ker \overline{V}$, respectively. Therefore consider the 
	subspace $\calU=\ker \overline{F}\cap \ker\overline{V}$ and denote by 
	$\mathbb{P}(\calU)$ the (classical) projective space over $\calU$. Consider 
	the map $\mathbb{P}(\calU)\to \Flag_{i,2g}(\fk)$, sending a subspace 
	$U\subset \calU$ to the flag
	\begin{equation*}
		W_0\subset W_1\subset\cdots\subset W_{i-1}\subset pr^{-1}(U).
	\end{equation*}
	With the conciderations above this map is easily seen to induce an 
	isomorphism of (classical) varieties $\mathbb{P}(\calU)\to 
	\zeta_i^{-1}\left( (W_j)_{j=0}^{i-1} \right)(\fk)$.  Hence the fibers of 
	$\zeta_i$ are connected. To see that they are nonempty we have to check that 
	$\dim\calU\geq 1$. This is automatic if $\overline{F}$ and $\overline{V}$ 
	are the zero morphism. By Proposition \ref{FVandPair}(\ref{imker}) we know that 
	$\im \overline{V}\subset \ker \overline{F}$ and $\im \overline{F}\subset\ker 
	\overline{V}$. Now the nilpotency of $\overline{F}$ implies that $\im 
	\overline{F}\cap \ker \overline{F}\neq 0$ if $\overline{F}\neq 0$ and
	the nilpotency of $\overline{V}$ implies that $\im \overline{V}\cap \ker 
	\overline{V}\neq 0$ if $\overline{V}\neq 0$, whence the claim.
\end{proof}

\begin{prop}
	Let $g\geq 1$ and $k\geq 0$. If $A\in\calA^{(k)}_g(\fk)$,
	the fiber $\pi^{-1}(A)$ consists of $\begin{pmatrix}g\\k\end{pmatrix}ON_{k}$ 
	connected components. In particular it is connected if and only if $k=0$.
\end{prop}

\begin{proof}
	Combine Proposition \ref{posto0red} and Proposition \ref{ConnectedFib}.
\end{proof}

\section{\texorpdfstring{Dimension of the fibers of $\pi$}{Dimension of the fibers of pi}}

Let $g=2$ or $g=3$ and let $A\in\calA_g(\fk)$. Depending on $w_A\in\Wfin$ we list the dimension of 
$\pi^{-1}(A)\subset\calA_I$. It can be read off the calculations in Sections 
\ref{Fib2} and \ref{Fib3} and the results of Section \ref{secposprank}.

\begin{center}
	\begin{tabular}{|c|c|}
		\hline
		\multicolumn{2}{|c|}{$g=2$}\\
		\hline
		\hline
		$\Wfin$&$\dim$\\
		\hline
		$\id$&1\\
		\hline
		$s_2$&1\\
		\hline
		$s_{12}$&0\\
		\hline
		$s_{212}$&0\\
		\hline
	\end{tabular}
	\begin{tabular}{|c|c||c|c|}
		\hline
		\multicolumn{4}{|c|}{$g=3$}\\
		\hline
		\hline
		$\Wfin$&$\dim$&$\Wfin$&$\dim$\\
		\hline
		$\id$&3&$s_{123}$&1\\
		\hline
		$s_3$&2&$s_{3123}$&1\\
		\hline
		$s_{23}$&2&$s_{23123}$&0\\
		\hline
		$s_{323}$&1&$s_{323123}$&0\\
		\hline
	\end{tabular}
\end{center}

\section{The KR stratification}

This section contains the results about the KR stratification on $\calA_I$ that 
we are going to use. We will use an ad hoc definition on $\fk$-valued points and 
we refer to \cite[Section 2.4]{KRDL} for a more comprehensive treatment of the 
subject.

\subsection{Relative positions}

\begin{prop}{\cite[Section 3]{SSLIW}}
	Let $g\geq 1$, $w\in \Wfin$ and $(W_i)_{i=0}^{2g}\in\AFlag_{w,2g}(\fk)$. 
	There is a unique element $x=t^\lambda\omega\in W_a\tau$ ($\omega\in W,\ 
	\lambda\in X_*(T)$) such that there is a basis $(\varepsilon_i)_{i=0}^{2g}$ 
	of $\fk^{2g}$ with the following properties:
	\begin{enumerate}
		\item $\lambda(i)\in\{0,1\}$ for all $i$.
		\item For every $i$, $W_i$ is spanned by 
			$\varepsilon_1,\ldots,\varepsilon_i$.
		\item If $V_w(W_{i-1})\subsetneqq V_w(W_{i})$ for any $i\geq 1$, we have 
			$V_w(W_{i})=V_w(W_{i-1})\oplus \fk\cdot \varepsilon_{\omega(i)}$.
		\item \begin{equation*}
				\im 
				V_w=\bigoplus_{\substack{i=1,\ldots,2g\\\lambda(i)=0}}\fk\cdot 
				\varepsilon_i.
			\end{equation*}
	\end{enumerate}
\end{prop}
We call any such basis a \emph{KR basis} for $(W_i)_{i=0}^{2g}$ and $x$ is 
called the \emph{KR type} of $(W_i)_{i=0}^{2g}$.

The set of possible KR types (as $w$ runs through $\Wfin$) is denoted by 
$\Adm$. It is a subset of $W_a\tau$. Given $w\in \Wfin$ and $x\in \Adm$ we denote by 
$\calL(x,w)$ the set of flags in $\AFlag_{w}(\fk)$ with KR type equal to $x$. 

\subsection{The KR stratification}

On $\calA_I$ we have the Kottwitz-Rapoport stratification (a stratification in the 
sense of Section \ref{prank})
\begin{equation*}
	\calA_I=\coprod_{x\in \Adm} \KR{x},
\end{equation*}
given by $(A_i)_i\in \KR{x}(\fk)$ if and only if $\iota_{A_0}((A_i)_i)\in 
\calL(x,w_{A_0})$.

The following Proposition lists some properties of the KR stratification.

\begin{prop}{\cite[Section 2.5]{SSLIW}}\label{KRprop}
	Let $x,y\in \Adm$ and $\omega\in W,\ \lambda\in 
	X_*(T)$ such that $x=t^\lambda \omega$. 
	\begin{enumerate}
		\item $\KR{x}$ is equi-dimensional of dimension $\ell(x)$.
		\item\label{KRp-rank} The $p$-rank is constant on $\KR{x}$ with value 
			$\#\mathrm{Fix}(\omega)/2$ (where 
			$\mathrm{Fix}(\omega)=\{i\in\{1,\ldots,2g\}\pipe \omega(i)=i\}$).
		\item We have $\KR{x}\subset\overline{\KR{y}}$ if and only if $x\leq y$.
		\item\label{krirred} If $\KR{x}$ is not contained in the supersingular locus $\calS_I$, 
			then $\KR{x}$ is irreducible.
	\end{enumerate}
\end{prop}

In view of property (\ref{KRp-rank}) we denote by $\Adm^{(i)}$ the set of 
admissible elements of $p$-rank $i$, $0\leq i\leq g$.

\begin{lem}{\cite[Lemma 8.1]{SSLIW}}\label{whatisAdmn}
	The projection $\widetilde{W}\to W$ induces a bijection $\xi:\Admn\to 
	\{\omega\in 
	W\pipe \mathrm{Fix}(\omega)=\emptyset\}$. Its inverse is given by $\omega\mapsto 
	t^{\lambda(\omega)}\omega$ with 
	\begin{equation*}
		\lambda(\omega)(i) = \left\{ 
		\begin{array}{ll}
			0, & \omega^{-1}(i) > i \\
			1, & \omega^{-1}(i) < i
		\end{array}
		\right., \quad i=1, \dots, 2g.
	\end{equation*}
\end{lem}

\subsection{\texorpdfstring{The set $\Admn$}{The set of admissible elements}}\label{Admn}

In \cite{Yu} Yu gives a list of all the 29 elements of $\Admn$ for $g=3$.
We copy this list as we will use it extensively.
\begin{center}
	\begin{tabular}{|c|l||c|l|}
		\hline
		KR&$(\lambda,w)\in X_*(T)\rtimes W$&KR&$(\lambda,w)\in X_*(T)\rtimes 
		W$\\
		\hline
		$\tau$&$(0,0,0,1,1,1),\ (14)(25)(36)$&$s_{310}\tau$&$(0,0,1,0,1,1),\ 
		(132645)$\\
		\hline
		$s_0\tau$&$(0,0,0,1,1,1),\ (1463)(25)$&$s_{120}\tau$&$(0,0,0,1,1,1),\ 
		(16)(2453)$\\
		\hline
		$s_1\tau$&$(0,0,0,1,1,1),\ (142635)$&$s_{320}\tau$&$(0,0,1,0,1,1),\ 
		(154623)$\\
		\hline
		$s_2\tau$&$(0,0,0,1,1,1),\ (153624)$&$s_{230}\tau$&$(0,1,0,1,0,1),\ 
		(124653)$\\
		\hline
		$s_3\tau$&$(0,0,1,0,1,1),\ (1364)(25)$&$s_{201}\tau$&$(0,0,0,1,1,1),\ 
		(1562)(34)$\\
		\hline
		$s_{10}\tau$&$(0,0,0,1,1,1),\ (145)(263)$&$s_{301}\tau$&$(0,0,1,0,1,1),\ 
		(135642)$\\
		\hline
		$s_{20}\tau$&$(0,0,0,1,1,1),\ (153)(246)$&$s_{121}\tau$&$(0,0,0,1,1,1),\ 
		(16)(25)(34)$\\
		\hline
		$s_{30}\tau$&$(0,0,1,0,1,1),\ 
		(13)(25)(46)$&$s_{231}\tau$&$(0,1,0,1,0,1),\ (1265)(34)$\\
		\hline
		$s_{01}\tau$&$(0,0,0,1,1,1),\ (142)(356)$&$s_{312}\tau$&$(0,0,1,0,1,1),\ 
		(16)(2354)$\\
		\hline
		$s_{21}\tau$&$(0,0,0,1,1,1),\ 
		(15)(26)(34)$&$s_{323}\tau$&$(0,1,1,0,0,1),\ (123654)$\\
		\hline
		$s_{31}\tau$&$(0,0,1,0,1,1),\ 
		(135)(264)$&$s_{3010}\tau$&$(0,0,1,0,1,1),\ (132)(456)$\\
		\hline
		$s_{12}\tau$&$(0,0,0,1,1,1),\ 
		(16)(24)(35)$&$s_{3120}\tau$&$(0,0,1,0,1,1),\ (16)(23)(45)$\\
		\hline
		$s_{32}\tau$&$(0,0,1,0,1,1),\ 
		(154)(236)$&$s_{3230}\tau$&$(0,1,1,0,0,1),\ (123)(465)$\\
		\hline
		$s_{23}\tau$&$(0,1,0,1,0,1),\ 
		(124)(365)$&$s_{2301}\tau$&$(0,1,0,1,0,1),\ (12)(34)(56)$\\
		\hline
		$s_{010}\tau$&$(0,0,0,1,1,1),\ (145632)$&&\\
		\hline
	\end{tabular}
\end{center}
\section{\texorpdfstring{KR strata and the fibers of $\pi$}{KR strata and the fibers of pi}}\label{KRandFib}

Let $g\geq 1$ and $x\in\Adm$. We write
\begin{equation*}
	\ES(x)=\{w\in \Wfin\pipe \pi^{-1}(EO_w)\cap A_{I,x}\neq \emptyset\}.
\end{equation*}
Then \cite[Corollary 3.3]{SSLIW} states that
\begin{equation}\label{eovskrbasis}
		\pi(\calA_{I,x})=\coprod_{w\in\ES(x)}EO_w.
\end{equation}
Hence in order to understand the relationship 
between the EO and the KR stratification we need to understand the sets 
$\ES(x)$.

Now for all $w\in\Wfin$ we have $w\in\ES(x)
\Leftrightarrow \calL(x,w)\neq\emptyset $ and it is therefore sufficient 
to study the sets $\calL(x,w)$. We will do this for $g=3$, using our 
calculations of the sets
$\AFlag_{w,6}$.  The sets $\calL(x,\id)$ are rather complicated and we content 
ourselves with determining whether they are nonempty. For the other final 
elements $w$ of $p$-rank 0 we are able to determine the sets $\calL(x,w)$ 
completely.

First we have the following general result:

\begin{lem}\label{KRgeneral}
	Let $g\geq 1$. 
	\begin{enumerate}
		\item\label{KRgeneralfirst} For $\omega\in 
			S_g=\left<s_1,\ldots,s_{g-1}\right>\subset W$ we have 
			$\omega\tau\in\Admn$ and $\ES(\omega\tau)=\id$.
		\item\label{KRgeneralsecond}For $x=t^\lambda \omega\in \Admn$ 
			($\lambda\in X_*(T),\ \omega\in W$) we write 
			$N_x=\{i\in\{1,\ldots,2g\}\pipe \omega^2(i)<\omega(i)<i\}$. Then for 
			$(A_i)_i\in\KR{x}(\fk)$ we have $g-a(A_0)\geq \# N_x$.
	\end{enumerate}
\end{lem}

\begin{proof}
	\begin{enumerate}
		\item Let $\omega\in S_g$, then $\omega\tau$ is admissible by Lemma 
			\ref{whatisAdmn} above. Consider $(A_i)_i\in\KR{\omega\tau}(\fk)$ 
			with image $(W_i)_i$ 
			under $\iota_{A_0}$. $\omega\tau$ satisfies
			$\xi(\omega\tau)(\{g+1,\ldots,2g\})=\{1,\ldots,g\}$, which means that $\im 
			V_{w_{A_0}}=\ker V_{w_{A_0}}=W_g$. By Proposition \ref{FVandPair}\ref{imker} this 
			implies that $\im F_{w_{A_0}}=\ker F_{w_{A_0}}$, hence the canonical filtration on 
			$\Dieu(A_0)$ is given by $0\subset F(\Dieu)\subset \Dieu$ which has 
			associated final element $\id$.
		\item By Lemma \ref{DimV2} and Proposition \ref{EOProp} our claim is 
			equivalent to the following statement: Let $w\in\Wfin^{(0)}$ and
			$(W_i)_{i=0}^{2g}\in\AFlag_{w,2g}(\fk)$ of KR type $x$. Then $\dim\im 
			V_w^2\geq \# N_x$. 
			
			But if $(\varepsilon)_{i=0}^{2g}$ is a KR basis for 
			$(W_i)_i$, the set $\{V(\varepsilon_{\omega(i)})\pipe i\in N_x\}$ is 
			a linearly independent subset of $\im V_w^2$ of 
			cardinality $\# N_x$.
	\end{enumerate}
\end{proof}

For the rest of this Section $g$ is equal to $3$.

\subsection{\texorpdfstring{$w=\id$}{w=id}}\label{DimI}

Let $x\in\Admn$. By Lemma \ref{KRgeneral}(\ref{KRgeneralsecond}) we know that $\calL(x,\id)\neq\emptyset$
implies that $N_x=\emptyset$. Inspecting the table in Section \ref{Admn} we see 
that this condition is only satisfied for $x\in \{\tau, s_1\tau, s_2\tau, 
s_{21}\tau, s_{12}\tau, s_{121}\tau, s_{30}\tau, s_{310}\tau, s_{320}\tau, 
s_{3120}\tau, s_{2301}\tau\}$. We claim that $\id\in\ES(x)$ for all $x$ in this set.
For $x\in S_3\tau=\{\tau, s_1\tau, s_2\tau, s_{21}\tau, s_{12}\tau, \allowbreak
s_{121}\tau\}$ this is true by Lemma \ref{KRgeneral}(\ref{KRgeneralfirst}). For the remaining 
elements we write down an explicit nonempty subset $\calK(x)\subset 
\calL(x,\id)$. 

\begin{center}
	\begin{longtable}{|c|c|c|}
		\hline
		$\calK(s_{30}\tau)$&$\calK(s_{3120}\tau)$&$\calK(s_{2301}\tau)$\\
		\hline
		$\begin{pmatrix}
			0&0&1\\
			0&0&0\\
			0&0&0\\
			0&0&0\\
			0&1&0\\
			1&0&0
		\end{pmatrix}$&
		$\begin{pmatrix}
			0&0&1\\
			0&0&0\\
			0&0&0\\
			0&0&0\\
			1&0&0\\
			0&1&0
		\end{pmatrix}$&
		$\begin{pmatrix}
			0&1&0\\
			0&0&0\\
			0&0&0\\
			0&0&0\\
			0&0&1\\
			1&0&0
		\end{pmatrix}$\\
		\hline
		\multicolumn{3}{c}{}\\
		\hline
		$\calK(s_{310}\tau)$&$\calK(s_{320}\tau)$&\\
		\hline
		$\begin{pmatrix}
			0&0&-1\\
			0&0&b_2\\
			0&0&b_2^{p+1}+b_1\\
			b_1&b_2&\alpha\\
			-b_2^p&1&0\\
			1&0&0
		\end{pmatrix}$&
		$\begin{pmatrix}
			0&0&-1\\
			0&0&b_2\\
			0&0&b_2^{p^{-1}}b_2+b_1\\
			b_1&b_2&\alpha\\
			-b_2^{p^{-1}}&1&0\\
			1&0&0
		\end{pmatrix}$&\\
		\hline
		$b_2\in\fk-\FF{p^2},\ 
		\alpha\in\fk$&$b_2\in\fk-\FF{p^2},\ \alpha\in\fk$&\\
		$b_1\in\fk$ a root of &$b_1\in\fk$ a root of&\\
		$T^p+T+b_2^{p(p+1)}\in\fk[T]$&$T^{p^2}+T^p+b_2^{p+1}\in\fk[T]$&\\
		\hline
		\hline
		\multicolumn{2}{|c|}{2}&\\
		\hline
	\end{longtable}
\end{center}

\subsection{\texorpdfstring{$w=s_3$}{w=s\_3}}\label{DimII}

We list those $\calL(x,s_3)$ which are nonempty.
\begin{center}
	\begin{tabular}{|c|c|c|}
		\hline
		$\calL(s_{120}\tau,s_3)$&$\calL(s_{3120}\tau,s_3)$&$\calL(s_{312}\tau,s_3)$\\
		\hline
		$\begin{pmatrix}
			0&0&0\\
			0&0&0\\
			0&0&0\\
			0&0&1\\
			1&0&0\\
			c&1&0
		\end{pmatrix}$ &
		$\begin{pmatrix}
			0&0&\alpha\\
			0&0&0\\
			0&0&0\\
			0&0&1\\
			1&0&0\\
			c&1&0
		\end{pmatrix}$ &
		$\begin{pmatrix}
			0&0&1\\
			0&0&0\\
			0&0&0\\
			0&0&0\\
			1&0&0\\
			c&1&0
		\end{pmatrix}$\\
		\hline
		$c\in\fk$&$c\in\fk,\ \alpha\in \fk^\times$&$c\in\fk$\\
		\hline
		\hline
		1&2&1\\
		\hline
	\end{tabular}
	\begin{longtable}{|c|c|c|}
		\hline
		$\calL(s_{201}\tau,s_3)$&$\calL(s_{2301}\tau,s_3)$&$\calL(s_{231}\tau,s_3)$\\
		\hline
		$\begin{pmatrix}
			0&0&0\\
			0&0&0\\
			0&0&0\\
			0&1&0\\
			0&c&1\\
			1&0&0
		\end{pmatrix}$ &
		$\begin{pmatrix}
			0&a&0\\
			0&0&0\\
			0&0&0\\
			0&1&0\\
			0&c&1\\
			1&0&0
		\end{pmatrix}$ &
		$\begin{pmatrix}
			0&1&0\\
			0&0&0\\
			0&0&0\\
			0&0&0\\
			0&c&1\\
			1&0&0
		\end{pmatrix}$\\
		\hline
		$c\in\fk$&$c\in\fk,\ a\in\fk^\times$&$c\in\fk$\\
		\hline
		\hline
		1&2&1\\
		\hline
		\multicolumn{3}{c}{}\\
		\hline
		$\calL(s_{30}\tau,s_3)$&$\calL(s_0\tau,s_3)$&$\calL(s_3\tau,s_3)$\\
		\hline
		$\begin{pmatrix}
			0&0&\alpha\\
			0&0&0\\
			0&0&0\\
			0&0&1\\
			0&1&0\\
			1&0&0
		\end{pmatrix}$ &
		$\begin{pmatrix}
			0&0&0\\
			0&0&0\\
			0&0&0\\
			0&0&1\\
			0&1&0\\
			1&0&0
		\end{pmatrix}$ &
		$\begin{pmatrix}
			0&0&1\\
			0&0&0\\
			0&0&0\\
			0&0&0\\
			0&1&0\\
			1&0&0
		\end{pmatrix}$\\
		\hline
		$\alpha\in\fk^\times$&&\\
		\hline
		\hline
		1&0&0\\
		\hline
	\end{longtable}
\end{center}

\subsection{\texorpdfstring{$w=s_{23}$}{w=s\_23}}\label{DimIII}

We list those $\calL(x,s_{23})$ which are nonempty.

\begin{center}
	\begin{longtable}{|c|c|c|c|c|}
		\hline
		$\calL(s_{20}\tau,s_{23})$&$\calL(s_{320}\tau,s_{23})$&$\calL(s_{120}\tau,s_{23})$&$\calL(s_{3120}\tau,s_{23})$&$\calL(s_{312}\tau,s_{23})$\\
		\hline
		$\begin{pmatrix}
			0&0&0\\
			0&0&0\\
			0&0&0\\
			0&0&1\\
			1&0&0\\
			0&1&0
		\end{pmatrix}$&
		$\begin{pmatrix}
			0&0&\alpha\\
			0&0&0\\
			0&0&0\\
			0&0&1\\
			1&0&0\\
			0&1&0
		\end{pmatrix}$&
		$\begin{pmatrix}
			0&0&0\\
			0&0&0\\
			0&0&0\\
			0&0&1\\
			1&0&0\\
			c&1&0
		\end{pmatrix}$&
		$\begin{pmatrix}
			0&0&\alpha\\
			0&0&0\\
			0&0&0\\
			0&0&1\\
			1&0&0\\
			c&1&0
		\end{pmatrix}$&
		$\begin{pmatrix}
			0&0&1\\
			0&0&0\\
			0&0&0\\
			0&0&0\\
			1&0&0\\
			c&1&0
		\end{pmatrix}$\\
		\hline
		&$\alpha\in\fk^\times$&$c\in\fk^\times$&$c,\alpha\in\fk^\times$&$c\in\fk^\times$\\
		\hline
		\hline
		0&1&1&2&1\\
		\hline
		\multicolumn{5}{c}{}\\
		\hline
		$\calL(s_{32}\tau,s_{23})$&$\calL(s_{310}\tau,s_{23})$&$\calL(s_{10}\tau,s_{23})$&$\calL(s_{31}\tau,s_{23})$&$\calL(s_{01}\tau,s_{23})$\\
		\hline
		$\begin{pmatrix}
			0&0&1\\
			0&0&0\\
			0&0&0\\
			0&0&0\\
			1&0&0\\
			0&1&0
		\end{pmatrix}$&
		$\begin{pmatrix}
			0&0&\alpha\\
			0&0&0\\
			0&0&0\\
			0&0&1\\
			0&1&0\\
			1&0&0
		\end{pmatrix}$&
		$\begin{pmatrix}
			0&0&0\\
			0&0&0\\
			0&0&0\\
			0&0&1\\
			0&1&0\\
			1&0&0
		\end{pmatrix}$&
		$\begin{pmatrix}
			0&0&1\\
			0&0&0\\
			0&0&0\\
			0&0&0\\
			0&1&0\\
			1&0&0
		\end{pmatrix}$&
		$\begin{pmatrix}
			0&0&0\\
			0&0&0\\
			0&0&0\\
			0&1&0\\
			1&0&0\\
			0&0&1
		\end{pmatrix}$\\
		\hline
		&$\alpha\in\fk^\times$&&\\
		\hline
		\hline
		0&1&0&0&0\\
		\hline
		\newpage
		\hline
		$\calL(s_{201}\tau,s_{23})$&$\calL(s_{23}\tau,s_{23})$&$\calL(s_{231}\tau,s_{23})$&$\calL(s_{3230}\tau,s_{23})$&$\calL(s_{3010}\tau,s_{23})$\\
		\hline
		$\begin{pmatrix}
			0&0&0\\
			0&0&0\\
			0&0&0\\
			0&1&0\\
			1&0&0\\
			0&c&1
		\end{pmatrix}$ &
		$\begin{pmatrix}
			0&1&0\\
			0&0&0\\
			0&0&0\\
			0&0&0\\
			0&0&1\\
			1&0&0
		\end{pmatrix}$ &
		$\begin{pmatrix}
			0&1&0\\
			0&0&0\\
			0&0&0\\
			0&0&0\\
			0&c&1\\
			1&0&0
		\end{pmatrix}$&
		$\begin{pmatrix}
			0&1&0\\
			0&0&1\\
			0&0&0\\
			0&0&0\\
			0&0&\beta\\
			1&0&0
		\end{pmatrix}$&
		$\begin{pmatrix}
			0&0&0\\
			0&0&0\\
			0&0&1\\
			0&1&0\\
			1&0&0\\
			0&0&\beta
		\end{pmatrix}$\\
		\hline
		$c\in\fk^\times$&&$c\in\fk^\times$&$\beta\in\fk$&$\beta\in\fk$\\
		\hline
		\hline
		1&0&1&1&1\\
		\hline
	\end{longtable}
\end{center}

\subsection{\texorpdfstring{$w=s_{323}$}{w=s\_323}}\label{DimIV}

We list those $\calL(x,s_{323})$ which are nonempty.

\begin{center}
	\begin{longtable}{|c|c|c|c|}
		\hline
		$\calL(s_{3230}\tau,s_{323})$&$\calL(s_{323}\tau,s_{323})$&$\calL(s_{230}\tau,s_{323})$&$\calL(s_{3010}\tau,s_{323})$\\
		\hline
		$\begin{pmatrix}
			0&1&0\\
			0&0&1\\
			0&0&0\\
			0&0&0\\
			0&0&\beta\\
			1&0&0
		\end{pmatrix}$&
		$\begin{pmatrix}
			0&1&0\\
			0&0&1\\
			0&0&0\\
			0&0&0\\
			0&0&0\\
			1&0&0
		\end{pmatrix}$&
		$\begin{pmatrix}
			0&1&0\\
			0&0&0\\
			0&0&0\\
			0&0&0\\
			0&0&1\\
			1&0&0
		\end{pmatrix}$&
		$\begin{pmatrix}
			0&0&\alpha\\
			0&0&0\\
			0&0&0\\
			0&0&1\\
			0&1&0\\
			1&0&0
		\end{pmatrix}$\\
		\hline
		$\beta\in\fk^\times$&&&$\alpha\in\fk^\times$\\
		\hline
		\hline
		1&0&0&1\\
		\hline
		\multicolumn{4}{c}{}\\
		\hline
		$\calL(s_{010}\tau,s_{323})$&$\calL(s_{301}\tau,s_{323})$&$\calL(s_{2301}\tau,s_{323})$&\\
		\hline
		$\begin{pmatrix}
			0&0&0\\
			0&0&0\\
			0&0&0\\
			0&0&1\\
			0&1&0\\
			1&0&0
		\end{pmatrix}$&
		$\begin{pmatrix}
			0&0&1\\
			0&0&0\\
			0&0&0\\
			0&0&0\\
			0&1&0\\
			1&0&0
		\end{pmatrix}$&
		$\begin{pmatrix}
			0&a&1\\
			0&0&0\\
			0&0&0\\
			0&0&0\\
			0&1&0\\
			1&0&0
		\end{pmatrix}$&\\
		\hline
		&&$a\in\fk^\times$&\\
		\hline
		\hline
		0&0&1&\\
		\hline
	\end{longtable}
\end{center}

\section{Proof of the results of Section \ref{KRandFib}}\label{krproof}

In order to illustrate the method we show that $\calK(s_{310}\tau)\subset 
\calL(s_{310}\tau,\id)$. 
Let $V=V_{\id}$. For an element $(W_i)_{i=0}^6$ of $\calK(s_{310}\tau)(\fk)$ choose elements
$b_2\in\fk-\FF{p^2}$, $\alpha\in\fk$ and $b_1\in\fk$ with $b_1^p+b_1+b_2^{p(p+1)}=0$ such 
that 
\begin{equation}
	(W_i)_{i=0}^6=\Phi
	\begin{pmatrix}
		0&0&-1\\
		0&0&b_2\\
		0&0&b_2^{p+1}+b_1\\
		b_1&b_2&\alpha\\
		-b_2^p&1&0\\
		1&0&0
	\end{pmatrix}\tag{$*$}.
\end{equation}
First we write down a matrix $C=(c_1c_2c_3c_4c_5c_6)\in\GL{6}(\fk)$ such that $W_i=\oplus_{j=1}^i 
\fk\cdot c_j$ for all $0\leq i\leq 6$. For the first three columns of $C$ we can 
use the columns of the matrix of Equation $(*)$ above. We find the other columns using the 
condition that $(W_i)_{i=0}^6$ is a symplectic flag, meaning that $c_4\perp c_1,c_2$ and 
$c_5\perp c_1$. Hence
\begin{equation*}
	C=\begin{pmatrix}
		0&0&-1&0&0&0\\
		0&0&b_2&0&1&0\\
		0&0&b_2^{p+1}+b_1&0&b_2^p&1\\
		b_1&b_2&\alpha&1&0&0\\
		-b_2^p&1&0&0&0&0\\
		1&0&0&0&0&0
	\end{pmatrix}
\end{equation*}
satisfies our requirements.

From this matrix we can read off the images $(V(W_i))_{i=0}^6$ using the explicit 
description of Section \ref{StaBas} and we need to find a basis 
$(\varepsilon_i)_{i=1}^6$ of $\fk^6$ such that $W_i=\oplus_{j=1}^i\fk\cdot 
\varepsilon_j$ and such that $V(W_i)$ is spanned by a subset of 
$\{\varepsilon_1,\ldots,\varepsilon_i\}$ for each $0\leq i\leq 6$. First we have 
$V(W_2)=0$. Now $b_2\notin \FF{p^2}$ implies that $V(W_3)\nsubseteq W_1$ and 
hence we can take $\varepsilon_2=V(c_3)$. The equation for $b_1$ implies that $V(W_5)\subset W_2$ 
which means that we can use $\varepsilon_1=c_1$ and $\varepsilon_4=c_4$.

This means that a KR basis $(\varepsilon_i)_{i=1}^6$ of $(W_i)$ is given by the columns of the following 
matrix 
\begin{equation*}
	\varepsilon=
	\begin{pmatrix}
		0&0&-1&0&0&0\\
		0&0&b_2&0&1&0\\
		0&0&b_2^{p+1}+b_1&0&b_2^p&1\\
		b_1&b_1^{p^{-2}}&\alpha&1&0&0\\
		-b_2^p&-b_2^{p^{-1}}&0&0&0&0\\
		1&1&0&0&0&0
	\end{pmatrix}.
\end{equation*}
Here the equation for $b_1$ is needed to see that 
$V(\varepsilon_3)=\varepsilon_2$. We have observed that
$V(W_1)=V(W_2)=0$, $V(W_3)=V(W_4)=\left<\varepsilon_2\right>$, 
$V(W_5)=\left<\varepsilon_1,\varepsilon_2\right>$ and 
$V(W_6)=\left<\varepsilon_1,\varepsilon_2,\varepsilon_4\right>$. Hence if 
$\lambda\in X_*(T)$ and $\omega\in W$ are such that
$(W_i)_i\in\calL(t^\lambda \omega,\id)$, we see that $\lambda=(0,0,1,0,1,1)$ and 
that $\omega(3)=2,\ \omega(5)=1$ and $\omega(6)=4$. The $\omega\in W$ satisfying 
these conditions is given by 
\begin{equation*}
	\omega=
	\begin{pmatrix}
		1&2&3&4&5&6\\
		3&6&2&5&1&4
	\end{pmatrix}
\end{equation*} and from the table in Section \ref{Admn} we see that $t^\lambda 
\omega=s_{310}\tau$.

The proof of $\calK(s_{320}\tau)\subset\calL(s_{320}\tau,\id)$ is similar and in 
all the other cases it is very easy to write down a suitable KR basis.
\section{\texorpdfstring{The sets $\ES(x)$ for $x\in\Admn$ in dimensions 2 and 
3}{The sets ES(x) in dimensions 2 and 3}}\label{ESset}

Let $g=2$. In this case the sets $\ES(x)$ for $x\in\Admn$ have already been 
known, see for instance \cite[Example 3.4]{SSLIW}. They are given by

\begin{center}
	\begin{tabular}{|l|cc|}
		\hline
		$x$&\multicolumn{2}{c|}{$\ES(x)$}\\
		\hline
		\hline
		$\tau,s_1\tau$&$\id$&\\
		\hline
		$s_2\tau$, $s_0\tau$&&$s_2$\\
		\hline
		$s_{20}\tau$&$\id$&$s_2$\\
		\hline
	\end{tabular}
\end{center}

Let $g=3$. The following table contains the sets $\ES(x)$ for $x\in\Admn$. They 
can be read off the calculations in Section \ref{KRandFib}. The upper block 
contains the supersingular elements.
\begin{center}
	\begin{tabular}{|l|cccc|}
		\hline
		$x$&\multicolumn{4}{c|}{$\ES(x)$}\\
		\hline
		\hline
		$\tau, s_1\tau, s_2\tau, s_{21}\tau, s_{12}\tau, s_{121}\tau$& 
		$\id$&&&\\
		\hline
		$s_3\tau, s_0\tau$& &$s_3$&&\\
		\hline
		$s_{30}\tau$&$\id$&$s_3$&&\\
		\hline
		\hline
		$s_{10}\tau, s_{23}\tau, s_{20}\tau, s_{31}\tau,s_{01}\tau, 
		s_{32}\tau$&&&$s_{23}$&\\
		\hline
		$s_{310}\tau, s_{320}\tau$&$\id$&&$s_{23}$&\\
		\hline
		$s_{3120}\tau$&$\id$&$s_3$&$s_{23}$&\\
		\hline
		$s_{120}\tau, s_{312}\tau,s_{201}\tau, s_{231}\tau$&&$s_3$&$s_{23}$&\\
		\hline
		$s_{010}\tau, s_{323}\tau, s_{301}\tau, s_{230}\tau$&&&&$s_{323}$\\
		\hline
		$s_{2301}\tau$&$\id$&$s_3$&&$s_{323}$\\
		\hline
		$s_{3010}\tau, s_{3230}\tau$&&&$s_{23}$&$s_{323}$\\
		\hline
	\end{tabular}
\end{center}

\begin{remark}
	We can use this table to answer a question posed in a preliminary version of 
	\cite{KRDL}: For every $g\geq 1$ one has the following inclusion:
	\begin{equation}\label{contained}
		\coprod_{\substack{x\in\Admn\\\KR{x}\subset\calS_I}}\KR{x}\subseteq\pi^{-1}\left(\coprod_{\substack{w\in\Wfin\\EO_w\subset 
		\calS_g}}EO_w\right).
	\end{equation}
	In loc.cit. it was asked whether this inclusion is an equality. The answer 
	is negative in the case $g=3$: Let $A\in EO_{\id}(\fk)$. By the table above 
	there is preimage $(A_i)_i\in(\pi^{-1}(A)\cap \KR{s_{310}\tau})(\fk)$ of 
	$A$, so that $(A_i)_i$ is contained in the right hand side of the above 
	inclusion, but as $\KR{s_{310}\tau}\nsubseteq \calS_I$ it is not contained 
	in the left hand side.
\end{remark}

\subsection{Some informal observations}

Let $1\leq g\leq 3$. It is interesting to note that $\ES\left(\xi^{-1}( \xi (x)^{-1})\right)=\ES(x)$ for every 
$x\in\Admn$. We don't know if this is true for arbitrary $g$.

Compare the line
\begin{center}
	\begin{tabular}{|l|c|}
		\hline
		$x$&$\ES(x)$\\
		\hline
		\hline
		$s_2\tau$, $s_0\tau$&$s_2$\\
		\hline
	\end{tabular}
\end{center}
of the table for $g=2$ with the lines
\begin{center}
	\begin{tabular}{|l|cc|}
		\hline
		$x$&\multicolumn{2}{c|}{$\ES(x)$}\\
		\hline
		\hline
		$s_3\tau, s_0\tau$&$s_3$&\\
		\hline
		$s_{10}\tau, s_{23}\tau, s_{20}\tau, s_{31}\tau,s_{01}\tau, 
		s_{32}\tau$&&$s_{23}$\\
		\hline
	\end{tabular}
\end{center}
of the table for $g=3$.

We can prove the following result, generalizing these lines:
Let $2\leq g$, $i\in\{0,1\}$ and consider the sets 
$S_j=\{s_j,s_{g-j}\}\subset W$ for $0\leq j\leq i$. 
Then for every element $x\in\Admn$ of the form 
$\left(t_{\upsilon(0)}\cdot t_{\upsilon(1)}\cdots t_{\upsilon(i)}\right)\tau$,
where $\upsilon\in S(\{0,\ldots,i\})$ and $t_j\in S_j$ for all $0\leq j\leq i$, we have 
$\ES(x)=\{s_{g-i}\cdot s_{g-i+1}\cdots s_{g}\}$. 

\section{The KR stratification and the supersingular locus}\label{krandss}

Let $g=3$. For $x\in\Admn$ we want to understand the intersection $\KR{x}\cap 
S_{I}$.

\begin{prop}
	Let $g=3$ and $x\in\Admn$. Then $\KR{x}\cap \calS_{I}=\emptyset\Leftrightarrow 
	\ES(x)=\{s_{23}\}$.
\end{prop}

\begin{proof}
	This follows from the results of Section \ref{EOandSS3}, using Equation 
	\ref{eovskrbasis}.
\end{proof}

\begin{remark}\label{KRvsDL}
	The relationship between the KR stratification and the supersingular locus 
	is closely related to the theory of affine Deligne-Lusztig varieties. In 
	\cite[Proposition 12.6]{haines}, Haines shows that for $x\in\Admn$ the nonemptiness of the 
	intersection $\KR{x}\cap \calS_I$ is equivalent to the nonemptiness of a 
	certain affine Deligne-Lusztig variety. In loc.cit. this result is stated 
	using $p$-adic Deligne-Lusztig varieties, but by \cite[Corollary 11.3.5]{affdl} 
	the non-emptiness of an affine Deligne-Lusztig variety is equivalent in the 
	function field and the $p$-adic case.
\end{remark}

We want to get a more precise statement about the intersection $\KR{x}\cap 
\calS_I$ in those cases where it is nonempty and not equal to $\KR{x}$. For this 
we need the following

\begin{prop}\label{DimProp}
	Let $f:X\to Y$ be a proper morphism of algebraic varieties over an 
	algebraically closed field $K$. Let $B\subset Y$ be a locally closed subset 
	equi-dimensional of dimension $d\in\mathbb{N}$. Let $A\subset X$ be a 
	locally closed subset with the property that there is a natural number 
	$e\in\mathbb{N}$ such that $f^{-1}(b)\cap A$ is
	irreducible of dimension $e$ and dense in $f^{-1}(b)\cap\overline{A}$ for 
	every $b\in B(K)$, where we denote by $\overline{A}$ the closure of $A$ in 
	$X$. Then $f^{-1}(B)\cap A$ is equi-dimensional of dimension 
	$d+e$. Furthermore the number of irreducible components of $f^{-1}(B)\cap A$ 
	is equal to the number of irreducible components of $B$.
\end{prop}

\begin{proof}
	We immediately reduce to the case $B=Y$ and $A=X$. We may also assume $Y$ to 
	be irreducible. Hence we are reduced to the statement of the following 
	Lemma whose proof we include for lack of reference.
\end{proof}

\begin{lem}
	Let $f:X\to Y$ be proper morphism of algebraic varieties over an 
	algebraically closed field $K$. Assume that $Y$ is irreducible of 
	dimension $d\in\mathbb{N}$ and that the fiber $f^{-1}(y)$ is irreducible of dimension 
	$e\in\mathbb{N}$ for every $y\in Y(K)$. Then $X$ is irreducible of dimension 
	$d+e$.
\end{lem}
\begin{proof}
	Let $X=C_1\cup C_2$ with closed subsets $C_1,C_2\subset X$. Let $f_1$ and 
	$f_2$ denote the restrictions of $f$ to $C_1$ and $C_2$ respectively. By a 
	Corollary to Chevalley's Theorem, see \cite[13.1.5]{EGA}, the sets 
	$F_i=\{y\in Y\pipe \dim f_i^{-1}(y)\geq e\}$ are closed subsets of $Y$, 
	$i=1,2$. For $y\in Y(K)$ the $e$-dimensional fiber $f^{-1}(y)$ is the union 
	of the closed subsets $f_1^{-1}(y)$ and $f_2^{-1}(y)$ and hence $y\in 
	F_1\cup F_2$. As the set $Y(K)$ is dense in $Y$ this implies that $Y=F_1\cup 
	F_2$ and by the irreducibility of $Y$ we may assume that $F_1=Y$. Let $x\in 
	X(K)$ be a closed point with image $y=f(x)\in Y(K)$. Then $\dim 
	f_1^{-1}(y)=\dim f^{-1}(y)\cap C_1=e$ and as $f^{-1}(y)\cap C_1$ is a closed 
	subset of $f^{-1}(y)$ and the latter is irreducible of dimension $e$, we get 
	$f^{-1}(y)\cap C_1=f^{-1}(y)$ and hence $C_1\supset f^{-1}(y)\ni x$. This 
	implies that $C_1=X$ and we see that $X$ is irreducible. Furthermore the 
	closed subset $\{y\in Y\pipe \dim f^{-1}(y)\geq e+1\}$ of $Y$ does not 
	contain a point of $Y(K)$, hence it is empty. This means that $\dim 
	f^{-1}(y)=e$ for all $y\in Y$. We can now apply \cite[10.6.1(iii)]{EGA} to 
	get the result.
\end{proof}

We are now ready to determine the dimension of the intersection $\KR{x}\cap 
\calS_I$ for those $x\in \Admn$ with $\KR{x}\nsubseteq \calS_I$. It is clear a 
priori that for any such $x$ we have $\dim\KR{x}\cap \calS_I\leq \dim\KR{x}-1=
\ell(x)-1$ as this intersection is a proper closed subset of the 
irreducible space $\KR{x}$, see Proposition \ref{KRprop}(\ref{krirred}). The following table 
shows that this inequality is in fact an equality for $g=3$. 

\begin{center}
	\begin{tabular}{|l||c|c|}
		\hline
		$x$&$\dim \KR{x}\cap \calS_I$&equi-dimensional?\\
		\hline
		$s_{310}\tau,\ s_{320}\tau$&2&?\\
		\hline
		$s_{3120}\tau$&3&?\\
		\hline
		$s_{2301}\tau$&3&?\\
		\hline
		$s_{120}\tau, s_{312}\tau,s_{201}\tau, s_{231}\tau$&2&$\surd$\\
		\hline
		$s_{010}\tau, s_{323}\tau, s_{301}\tau, s_{230}\tau$&2&$\surd$\\
		\hline
		$s_{3010}\tau, s_{3230}\tau$&3&$\surd$\\
		\hline
	\end{tabular}
\end{center}

\subsection{Proof}

For $x\in\{s_{120}\tau, s_{312}\tau,s_{201}\tau, s_{231}\tau\}$ we apply 
Proposition \ref{DimProp} for $\pi$ with $B=EO_{s_3}$ and $A=\KR{x}$. It is 
clear from the results of Section \ref{DimII} that the conditions on the fibers 
(appearing in Proposition \ref{DimProp}) are indeed satisfied. For example we have 
\begin{equation*}
	\overline{\KR{s_{120}\tau}}=\KR{s_{120}\tau}\cup\KR{s_{12}\tau}\cup\KR{s_{20}\tau}\cup\KR{s_{10}\tau}\cup\KR{s_{1}\tau}\cup\KR{s_{2}\tau}\cup\KR{s_{0}\tau}\cup\KR{\tau},
\end{equation*}
hence we see from Section \ref{DimII} that for $b\in EO_{s_3}(\fk)$ we have
\begin{equation*}
	\overline{\KR{s_{120}\tau}}\cap\pi^{-1}(b)=\KR{s_{120}\tau}\cap\pi^{-1}(b)\cup\KR{s_0\tau}\cap\pi^{-1}(b).
\end{equation*}
By the results of Section \ref{DimII} we see that
\begin{center}
	\setlength{\extrarowheight}{4pt}
	\begin{tabular}{|c|c|}
		\hline
		$\iota_b(\pi^{-1}(b)\cap\KR{s_{120}\tau})$&$\iota_b(\pi^{-1}(b)\cap\overline{\KR{s_{120}\tau}})$\\
		\hline
		\setlength{\extrarowheight}{0pt}
		$\begin{pmatrix}
			0&0&0\\
			0&0&0\\
			0&0&0\\
			0&0&1\\
			1&0&0\\
			c&1&0
		\end{pmatrix}$&
		\setlength{\extrarowheight}{0pt}
		$\begin{pmatrix}
			0&0&0\\
			0&0&0\\
			0&0&0\\
			0&0&1\\
			1&0&0\\
			c&1&0
		\end{pmatrix}\vee
		\begin{pmatrix}
			0&0&0\\
			0&0&0\\
			0&0&0\\
			0&0&1\\
			0&1&0\\
			1&0&0
		\end{pmatrix}$\\
		\hline
		\multicolumn{2}{|c|}{$c\in\fk$}\\
		\hline
	\end{tabular}
\end{center}
Hence $\pi^{-1}(b)\cap\KR{s_{120}\tau}$ is irreducible of dimension 1 and dense 
in $\pi^{-1}(b)\cap\overline{\KR{s_{120}\tau}}$. As $\pi^{-1}(EO_{s_3})\cap 
\KR{x}=\calS_I\cap\KR{x}$ by the results of Section \ref{EOandSS3} and the table 
in Section \ref{ESset}, our claim follows.

For $x\in\{s_{010}\tau, s_{323}\tau, s_{301}\tau, s_{230}\tau,s_{3010}\tau, 
s_{3230}\tau\}$ we apply Proposition \ref{DimProp} for $\pi$ with $B=\calS_3\cap 
EO_{s_{323}}$ and $A=\KR{x}$. It follows from Section \ref{DimIV} that the conditions on 
the fibers are indeed satisfied and we have 
$\pi^{-1}(EO_{s_{323}}\cap\calS_3)\cap\KR{x}=\calS_I\cap\KR{x}$ by the results 
of Section \ref{EOandSS3} and the table in Section \ref{ESset}.

Furthermore we have 
$\pi^{-1}(EO_{s_3})\cap\KR{s_{3120}\tau}\subset\calS_I\cap\KR{s_{3120}\tau}$ and
$\pi^{-1}(EO_{s_{323}}\cap\calS_3)\cap\KR{s_{2301}\tau}\subset\calS_I\cap\KR{s_{2301}\tau}$ and we use 
Proposition \ref{DimProp} and the results of Section \ref{DimII} and 
\ref{DimIV}, respectively, to see that these subsets have dimension $3$. 

Finally let $A\in EO_\id(\fk)$. Then $\pi^{-1}(A)\cap \KR{s_{310}\tau}\subset 
\calS_I\cap \KR{s_{310}\tau}$ and $\pi^{-1}(A)\cap \KR{s_{320}\tau}\subset 
\calS_I\cap \KR{s_{320}\tau}$. But these subsets have dimension at least 2
because $\dim\calK({s_{310}\tau})=\dim\calK({s_{320}\tau})=2$ (see Section 
\ref{DimI}).
\begin{remark}
	If $g$ is even it is shown in \cite[Proposition 8.12]{SSLIW} that every 
	top-dimensional irreducible component of $\calS_I$ is an irreducible 
	component of the left hand side of Equation \ref{contained}. Looking at the 
	table above we see that the corresponding statement is not true for $g=3$, as 
	$\dim S_I=3$ in this case.
\end{remark}

\begin{remark}
	It is strongly expected that the relationship mentioned in Remark 
	\ref{KRvsDL} extends to other properties of the intersection $\KR{x}\cap 
	\calS_I$, $x\in\Admn$. In particular strong evidence suggests that $\KR{x}\cap 
	\calS_I$ is equi-dimensional of dimension $n$ if and only if the 
	corresponding affine Deligne-Lusztig variety (in the function field case) is 
	equi-dimensional of dimension $n$, $n\in\mathbb{N}$. In \cite{He}, Goertz 
	and He explain a reduction method for affine Deligne-Lusztig varieties over 
	function fields which is completely analogous to the classical reduction 
	method by Deligne and Lusztig. Using this reduction method one sees that the 
	affine Deligne-Lusztig varieties corresponding to the intersections $\KR{x}\cap 
	\calS_I$ for $g=3$ and $x\in\{s_{310}\tau, s_{320}\tau, s_{3120}\tau, 
	s_{2301}\tau\}$ are equi-dimensional. Hence we expect that the 
	same is true for the intersections themselves.
\end{remark}

\providecommand{\bysame}{\leavevmode\hbox to3em{\hrulefill}\thinspace}
\providecommand{\MR}{\relax\ifhmode\unskip\space\fi MR }
\providecommand{\MRhref}[2]{%
  \href{http://www.ams.org/mathscinet-getitem?mr=#1}{#2}
}
\providecommand{\href}[2]{#2}


\begin{thebibliography}{10}

\bibitem{BBM}
Pierre Berthelot, Lawrence Breen, and William Messing, \emph{Th\'eorie de
  {D}ieudonn\'e cristalline. {II}}, Lecture Notes in Mathematics, vol. 930,
  Springer-Verlag, Berlin, 1982.

\bibitem{Dem}
Michel Demazure, \emph{Lectures on {$p$}-divisible groups}, Lecture Notes in
  Mathematics, Vol. 302, Springer-Verlag, Berlin, 1972.

\bibitem{Geer}
Torsten Ekedahl and Gerard van~der Geer, \emph{Cycle classes of the {E}-{O}
  stratification on the moduli of abelian varieties},
  \href{http://xxx.lanl.gov/abs/math.AG/0412272}{{\tt arXiv:math/0412272v2
  [math.AG]}}.

\bibitem{affdl}
Ulrich G{\"o}rtz, Thomas~J. Haines, Robert~E. Kottwitz, and Daniel~C. Reuman,
  \emph{Affine {D}eligne-{L}usztig varieties in affine flag varieties},
  \href{http://arxiv.org/abs/0805.0045}{{\tt arXiv:0805.0045v3 [math.AG]}}.

\bibitem{He}
Ulrich G{\"o}rtz and Xuhua He, \emph{Dimension of affine {D}eligne-{L}usztig
  varieties in affine flag varieties}, In preparation.

\bibitem{Hoeve}
Ulrich G{\"o}rtz and Maarten Hoeve, \emph{Ekedahl-{O}ort strata and
  {K}ottwitz-{R}apoport strata}, \href{http://arxiv.org/abs/0808.2537}{{\tt
  arXiv:0808.2537v1 [math.AG]}}.

\bibitem{KRDL}
Ulrich G{\"o}rtz and Chia-Fu Yu, \emph{Supersingular {K}ottwitz-{R}apoport
  strata and {D}eligne-{L}usztig varieties}, J. Inst. Math. Jussieu, To appear.

\bibitem{SSLIW}
\bysame, \emph{The supersingular locus of {S}iegel modular varieties with
  {I}wahori level structure}, \href{http://arXiv.org/abs/0807.1229v2}{{\tt
  arXiv:0807.1229v2 [math.AG]}}.

\bibitem{EGA}
A.~Grothendieck, \emph{\'{E}l\'ements de g\'eom\'etrie alg\'ebrique. {IV}.
  \'{E}tude locale des sch\'emas et des morphismes de sch\'emas. {III}},
  no.~28, 1966.

\bibitem{haines}
Thomas~J. Haines, \emph{Introduction to {S}himura varieties with bad reduction
  of parahoric type}, Harmonic analysis, the trace formula, and {S}himura
  varieties, Clay Math. Proc., vol.~4, Amer. Math. Soc., Providence, RI, 2005,
  pp.~583--642.

\bibitem{Harashita}
Shushi Harashita, \emph{Ekedahl-{O}ort strata contained in the supersingular
  locus in the moduli space of abelian varieties}, Preprint.

\bibitem{LiOort}
Ke-Zheng Li and Frans Oort, \emph{Moduli of supersingular abelian varieties},
  Lecture Notes in Mathematics, vol. 1680, Springer-Verlag, Berlin, 1998.

\bibitem{Oda}
Tadao Oda, \emph{The first de {R}ham cohomology group and {D}ieudonn\'e
  modules}, Ann. Sci. \'Ecole Norm. Sup. (4) \textbf{2} (1969), 63--135.

\bibitem{Oort}
Frans Oort, \emph{A stratification of a moduli space of abelian varieties},
  Moduli of abelian varieties ({T}exel {I}sland, 1999), Progr. Math., vol. 195,
  Birkh\"auser, Basel, 2001, pp.~345--416.

\bibitem{Richartz}
Melanie Richartz, \emph{Klassifikation von selbstdualen {D}ieudonn\'egittern in
  einem dreidimensionalen polarisierten supersingul\"aren {I}sokristall},
  Bonner Mathematische Schriften, 311, Universit\"at Bonn Mathematisches
  Institut, 1998, Dissertation.

\bibitem{Yu}
Chia-Fu Yu, \emph{Kottwitz-{R}apoport strata of the {S}iegel moduli spaces},
  Taiwanese J. Math., To appear.

\bibitem{Yu2}
\bysame, \emph{Irreducibility and {$p$}-adic monodromies on the {S}iegel moduli
  spaces}, Adv. Math. \textbf{218} (2008), no.~4, 1253--1285.

\end{thebibliography}
\end{document}